\documentclass[10pt,reqno]{amsart}
\usepackage{amssymb,mathrsfs,graphicx}
\usepackage{ifthen}
\usepackage[margin=1in]{geometry}
\usepackage{caption}
\usepackage{sidecap}
\usepackage{rotating,dsfont}

\usepackage{colortbl}
\definecolor{black}{rgb}{0.0, 0.0, 0.0}
\definecolor{red}{rgb}{1.0, 0.5, 0.5}
\provideboolean{shownotes} % define a boolean variable
\setboolean{shownotes}{true} % defined as ``true'' or ``false''
\newcommand{\margnote}[1]{
\ifthenelse{\boolean{shownotes}}%
{\marginpar{\raggedright\tiny\texttt{#1}}}%
{}%
}
\newcommand{\hole}[1]{
\ifthenelse{\boolean{shownotes}}%
{\begin{center} \fbox{ \rule {.25cm}{0cm} \rule[-.1cm]{0cm}{.4cm}
\parbox{.85\textwidth}{\begin{center} \texttt{#1}\end{center}} \rule
{.25cm}{0cm}}\end{center}} {} }

\title[On well/ill-posedness of $\alpha$-SQG equations in H\"older spaces]{On well/ill-posedness for the generalized surface quasi-geostrophic equations in H\"older spaces}

\author[Choi]{Young-Pil Choi}
\address[Young-Pil Choi]{\newline Department of Mathematics\newline
Yonsei University, 50 Yonsei-Ro, Seodaemun-Gu, Seoul 03722, Republic of Korea}
\email{ypchoi@yonsei.ac.kr}

\author[Jung]{Jinwook Jung}
\address[Jinwook Jung]{\newline Department of Mathematics and Research Institute for Natural Sciences \newline
Hanyang University, 222 Wangsimni-ro, Seongdong-gu, Seoul 04763, Republic of Korea}
\email{2jwook12@gmail.com}

\author[Kim]{Junha Kim}
\address[Junha Kim]{\newline Department of Mathematics
\newline Ajou University, 206 Worldcup-ro, Yeongtong-gu, Suwon 16499, Republic of Korea}
\email{junha02@ajou.ac.kr}

\numberwithin{equation}{section}

\newtheorem{theorem}{Theorem}[section]
\newtheorem{lemma}{Lemma}[section]
\newtheorem{corollary}{Corollary}[section]

\newtheorem{remark}{Remark}[section]

\newcommand{\calC}{\mathcal C}

\newcommand{\R}{\mathbb R}

\newcommand{\N}{\mathbb N}
\newcommand{\ls}{\lesssim}

\newcommand{\mc}{\mathcal C}

\newcommand{\bq}{\begin{equation}}
\newcommand{\eq}{\end{equation}}
\newcommand{\e}{\varepsilon}
\newcommand{\lt}{\left}
\newcommand{\rt}{\right}

\newcommand{\pa}{\partial}

\newcommand{\intr}{\int_{\R^2}}

\makeatletter
\def\moverlay{\mathpalette\mov@rlay}
\def\mov@rlay#1#2{\leavevmode\vtop{%
   \baselineskip\z@skip \lineskiplimit-\maxdimen
   \ialign{\hfil$\m@th#1##$\hfil\cr#2\crcr}}}
\newcommand{\charfusion}[3][\mathord]{
    #1{\ifx#1\mathop\vphantom{#2}\fi
        \mathpalette\mov@rlay{#2\cr#3}
      }
    \ifx#1\mathop\expandafter\displaylimits\fi}
\makeatother

\begin{document}
%%%%%%%%%%%%%%%%
\allowdisplaybreaks

\date{\today}

%\subjclass{2010 MSC: 	35Q35,	35Q92, 76T10} 

\begin{abstract} We establish the well/ill-posedness theories for the inviscid $\alpha$-surface quasi-geostrophic ($\alpha$-SQG) equations in H\"older spaces, where $\alpha  = 0$ and $\alpha = 1$ correspond to the two-dimensional Euler equation in the vorticity formulation and SQG equation of geophysical significance, respectively. We first prove the local-in-time well-posedness of $\alpha$-SQG equations in $\calC([0,T);\calC^{0,\beta}(\R^2))$ with $\beta \in (\alpha,1)$ for some $T>0$. We then analyze the strong ill-posedness in $\calC^{0,\alpha}(\R^2)$ constructing smooth solutions to the $\alpha$-SQG equations that exhibit $\calC^{0,\alpha}$--norm growth in a short time. In particular, we develop the nonexistence theory for $\alpha$-SQG equations in $\calC^{0,\alpha}(\R^2)$.
\end{abstract}

\maketitle \centerline{\date}

\tableofcontents

%%%%%%%%%%%%%%%%%%%%%%%%%%%%%%%%%%%%%%%%%%%%%%%%%%%%%%%%%%%%%%%%%%%%%%%%%%%%%%%%%5
%
%
%                        Section: Introduction 
%
%
%%%%%%%%%%%%%%%%%%%%%%%%%%%%%%%%%%%%%%%%%%%%%%%%%%%%%%%%%%%%%%%%%%%%%%%%%%%%%%%%%
\section{Introduction}\label{sec:intro}
In this paper, we consider the following inviscid $\alpha$-surface quasi-geostrophic ($\alpha$-SQG) equations in two dimensions:
\begin{align}\label{main_eq}
\begin{aligned}
&\pa_t \theta + u \cdot \nabla \theta = 0, \quad x \in \R^2, \ t > 0,\cr
& u = \nabla^\perp (-\Delta)^{-1 + \frac\alpha2}\theta
\end{aligned}
\end{align}
subject to the initial data
\bq\label{main_ini}
\theta(0,x) = \theta_0(x), \quad x \in \R^2
\eq
for $0 \leq \alpha \leq 2$, where $\nabla^\perp$ denotes the perpendicular gradient, i.e. $\nabla^\perp = (\pa_{x_2}, - \pa_{x_1})$. Here, the velocity field $u$ can be written as
\[
u(t,x) = K_\alpha \star \theta(t,x) = c_\alpha \intr \frac{(x-y)^\perp}{|x-y|^{2+\alpha}} \theta(t,y)\,dy
\]
for some $c_\alpha > 0$. For simplicity, we set $c_\alpha =1$ throughout the paper. The cases $\alpha  = 0$ and $\alpha = 1$ correspond to the two-dimensional Euler equations in the vorticity formulation and SQG equations of geophysical significance, respectively.

In the last decade, there have been significant developments in the ill-posedness and well-posedness theories of the $\alpha$-SQG equations.  To put our study in the proper perspective, we first recall a few of the references from the considerable amount of literature available on the well/ill-posedness theories within the framework of the $H^s$ Sobolev space. The local-in-time well-posedness in $H^s(\R^2)$ spaces with $s > \alpha+1$ is established in \cite{CCCGW12} for $\alpha \in (0,2)$. In \cite{CW12} and \cite{JLM22}, the local well-posedness of the logarithmic inviscid regularization of SQG equations is obtained in the borderline Sobolev space $H^{\alpha + 1}(\R^2)$ with $\alpha \in [0,1]$ and $\alpha \in (1,2)$, respectively. More recently, in \cite{CM22, JK24}, strong ill-posedness in the borderline space is studied in the case of $\alpha = 1$, see also \cite{BSV19}. Even though there have been rigorous constructions of non-trivial global-in-time solutions % in $H^s(\R^2)$ for some $s > \alpha+1$ 
\cite{ADdMW21, CQZZ23, CCG19, CCG20, GS17, HH15, HK21, HHH16}, the question of whether finite-time singularities emerge from initial data in $H^s(\R^2)$ for $s > \alpha+1$ remains open. Beyond the Sobolev spaces, there are results on ill-posedness and well-posedness theories in $\calC^{k,\beta}(\R^2)$ spaces. Local well-posedness in $\calC^{k,\beta} \cap L^q$ with $k \geq 1$, $\beta \in (0, 1)$, with $q > 1$ has been established for SQG \cite{Wu05}, with recent improvements removing the $L^q$ regularity requirement \cite{ACEK23}. On the other hand, it was established by \cite{JK24} and \cite{CM22} that the SQG equations $(\alpha = 1)$ is strong ill-posedness in $\calC^1(\mathbb{R}^2)$ and $\calC^k(\mathbb{R}^2)$ with $k \geq 2$, respectively. After that, the strong ill-posedness in $\calC^{k,\beta}$ with $k \geq 1$, $\beta \in (0,1]$, and $k+\beta > \alpha$ for the $\alpha$-SQG equations \eqref{main_eq} with $\alpha \in (1,2)$ is analyzed in \cite{CMpre}. 

Despite these developments, to the best of our knowledge, the ill-posedness and well-posedness theories for the $\alpha$-SQG equations \eqref{main_eq} have not been well explored in the H\"older spaces $\calC^{0,\beta}(\R^2)$, see \cite{JKYpre} where the well-posedness of \eqref{main_eq} in the half-plane is studied in certain weighted anisotropic H\"older spaces. In the current work, we establish the well/ill-posedness for the $\alpha$-SQG equations \eqref{main_eq} with $\alpha \in (0,1)$. 

Our first result shows the solution to \eqref{main_eq} is locally well-posed in the H\"older space $\calC^{0,\beta}(\R^2)$ with $\beta \in (\alpha,1)$.

\begin{theorem}[Well-posedness in $\calC^{0,\beta}(\R^2)$ with $\beta \in (\alpha,1)$]\label{thm_main} Suppose that the initial data $\theta_0$ satisfies
\[
\theta_0 \in L^1 \cap \calC^{0,\beta}(\R^2), \quad \beta \in (\alpha,1).
\]
Then there exists $T>0$, depending only on $\|\theta_0\|_{L^1 \cap \calC^{0,\beta}(\R^2)}$ and $\alpha$, and a unique solution $\theta \in \calC([0,T);  L^1 \cap \calC^{0,\beta}(\R^2))$ to  \eqref{main_eq}-\eqref{main_ini}, which is stable in the little H\"older space $\calC([0,T);c^{0,\beta}(\R^2))$ with compact supports,  in the sense of distributions. Moreover, if $\theta_1$ and $\theta_2$ are two such solutions on the time interval $[0,T]$ corresponding to the initial data $\theta_{1,0}$ and $\theta_{2,0}$, respectively, then we have
\[
\sup_{0 \leq t \leq T}\|(\theta_1 - \theta_2)(t)\|_{\dot H^{\alpha - 1}} \leq \|\theta_{1,0} - \theta_{2,0}\|_{\dot H^{\alpha - 1}}
\]
for some $C>0$. 
\end{theorem}
\begin{remark}Note that $L^1 \cap L^\infty(\R^2) \hookrightarrow  \dot{H}^{\alpha-1}(\R^2)$. Indeed, we have
\[\begin{aligned}
\|\theta\|_{\dot{H}^{\alpha-1}}^2 &= \intr \theta \Lambda^{-2+2\alpha} \theta\,dx\\
&= \intr \theta K_{2(1-\alpha)}\star\theta\,dx\\
&= \lt(\intr\int_{\{|x-y|\le 1\}} + \intr\int_{\{|x-y|\ge 1\}}\rt) \frac{\theta(x)\theta(y)}{|x-y|^{2\alpha}}\,dydx\\
&\le C \intr |\theta(x)| \| |x|^{-2\alpha} \mathds{1}_{|x|\le 1}\|_{L^1}\|\theta\|_{L^\infty}\,dx + C\|\theta\|_{L^1}^2\\
&\le C\|\theta\|_{L^1 \cap L^\infty}^2
\end{aligned}\]
for some $C>0$, where $\Lambda = (-\Delta)^\frac12$.
\end{remark}

\begin{remark}Clearly, the well-posedness theory established in Theorem \ref{thm_main} holds for $\beta = 1$.
\end{remark}

 The proof of Theorem \ref{thm_main} relies on the Lagrangian approach. For this, we need to show that the characteristics associated with \eqref{main_eq} are well-defined. Note that the velocity fields $u$ can be regarded as $\Lambda^{-(1-\alpha)}\theta$ in terms of regularity. This gives $\nabla \Lambda^{-(1-\alpha)}\theta \sim |\nabla|^\alpha \theta$ and thus it seems impossible to obtain the Lipschitz or log-Lipschitz continuity of the velocity fields, which provides the well-definedness of the characteristics, by taking into account the only bounded solution $\theta$. This observation highlights that the H\"older continuous solution $\theta$ with exponent $\alpha$ is in the borderline solution space. By investigating a cancellation structure for the Lipschitz estimate of $\Lambda^{-(1-\alpha)}\theta$, we demonstrate that the velocity fields $u$ is log-Lipscthiz continuous when $\theta \in \calC^{0,\alpha}(\R^2)$ and Lipschitz continuous when $\theta \in \calC^{0,\beta}(\R^2)$ with $\beta \in (\alpha,1)$. In particular, we find that the H\"older regularity of solutions is propagated in time when the velocity fields are Lipschitz. Using classical approximation arguments, we then construct the H\"older continuous solution $\theta \in \calC^{0,\beta}(\R^2)$ with $\beta \in (\alpha,1)$. Regarding the uniqueness of solutions, direct estimates of solutions in the H\"older space present technical challenges. Thus, we navigate those difficulties by dealing with the negative Sobolev space, as stated in Theorem \ref{thm_main}, and employing a {\it modulation} technique, inspired by \cite{CJ1}. Finally, establishing well-posedness necessitates ensuring solution stability, for which we utilize the little H\"older space, which is the $\calC^\infty(\R^2)$ closure of the usual H\"older norm.

From the result of Theorem \ref{thm_main}, we naturally wonder whether the solution is well-posed or ill-posed in the {\it borderline} H\"older space $\calC^{0,\alpha}(\R^2)$. Our second main result demonstrates that the initial value problem \eqref{main_eq}-\eqref{main_ini} is strongly ill-posed in $\calC^{0,\alpha}(\R^2)$. In this respect, our well-posedness theory, Theorem \ref{thm_main}, is sharp. 

\begin{theorem}[Ill-posedness in $\calC^{0,\alpha}(\R^2)$]\label{thm:inf}
	For any positive constants $\varepsilon$, $\delta$, and $M$, there exists $\theta_{0} \in  \calC^\infty_c(\R^2) $ satisfying $\|\theta_0\|_{\calC^{0,\alpha}} < \varepsilon$ such that the unique local-in-time smooth solution $\theta$ to \eqref{main_eq} with initial data $\theta_{0}$ blows up at some time $t^* \in (0,\delta)$, or the solution $\theta$ exists on the interval $[0,\delta]$ and satisfies
		$$	\sup_{t\in[0,\delta^*]} \| \theta(t,\cdot)\|_{\calC^{0,\alpha}}>M $$ for some $\delta^* \in (0,\delta]$.
\end{theorem} 

%{\color{blue}
%\begin{remark} Quantified estimate of $\calC^{0,\alpha}$--norm inflation is provided in Section~\ref{sec_pf}.
%\end{remark}
%}

\begin{remark} The same result with Theorem \ref{thm:inf} can be proved in the periodic case using a similar approach.
\end{remark}

As previously mentioned, while achieving log-Lipschitz continuity of the velocity fields ensures the well-definedness of associated characteristics, it alone does not suffice to propagate the H\"older regularity of solutions, thereby failing to prevent solutions from exhibiting $\calC^{0,\alpha}$--norm inflation. For the proof of Theorem \ref{thm:inf}, motivated from \cite{JK24}, we first construct a family of odd-odd symmetric initial data $\{ \theta_0^{(N)} \}_{N \geq n_0} \subset \calC^{\infty}_c(\mathbb{R}^2)$ consisting of disjoint {\it bubbles}. These serve as the foundation for observing strong norm inflation in the corresponding solution. For precise details regarding the setup of the initial data, see Section \ref{sec_ini} below. This together with the uniqueness of solutions from Theorem \ref{thm_main} shows that the solution stemming from $\theta_0^{(N)}$ retains odd-odd symmetry as long as it exists. We then provide velocity estimates for odd-odd solutions, which play a crucial role in capturing the norm inflation of a family of solutions under the {\it hyperbolic flow scenario}. This particular velocity approximation is initially introduced in \cite[Lemma 3.1]{KS14} for the two-dimensional Euler equation and has since been utilized in various instances to establish the growth of vorticity \cite{BL14, E21, EJ17, GP21, HK21, JKpre, JK22, JKYpre, KRYZ16, K21, X16, Z15, Zpre}. Finally, we study short-time dynamics of solutions to complete the proof of Theorem \ref{thm:inf}. It is worth noting that all estimates in the proof of Theorem \ref{thm:inf} are provided with explicit dependence on $N$. As a direct consequence, letting $N \to \infty$ in the initial data $\{ \theta_0^{(N)} \}_{N \geq n_0}$ yields the following result.

\begin{corollary}[Nonexistence of odd-odd solutions in $\calC^{0,\alpha}(\R^2)$]\label{nonexist}
	For any $\varepsilon > 0$, there exists $\theta_{0} \in \calC_c^{0,\alpha}(\R^2)$ satisfying \begin{equation*}
		\begin{split}
			\| \theta_0 \|_{\calC^{0,\alpha}} < \varepsilon 
		\end{split}
	\end{equation*} such that there is no odd-odd symmetric solution to \eqref{main_eq} with initial data $\theta_{0}$ belonging to $L^\infty([0,\delta];\calC^{0,\alpha}(\R^{2}))$ with any $\delta>0$. 
\end{corollary}

%\begin{remark} This result deals only with the odd-odd symmetric solutions. Lemma \ref{key_lem} below which requires the odd-odd symmetry is a crucial part to show $\calC^{0,\alpha}$--norm inflation of solutions and the uniqueness of $\calC^{0,\alpha}$-solutions is unknown to the best of our knowledge.
%\end{remark}

\begin{remark} The initial data $\theta_0$ in Corollary~\ref{nonexist} is given by the aforementioned $\{ \theta_0^{(N)} \}_{N \geq n_0}$ with $N = \infty$. It can be proved that $\theta_0$ is in the little H\"older space $c^{0,\alpha}(\R^2)$.% by the use of $\beta>0$.
\end{remark}

\subsubsection*{Notations} All generic positive constants are denoted by $C$. $f \ls g$ represents that there exists a positive constant $C>0$ such that $f \leq C g$. $A \simeq B$ stands for $A \ls B$ and $B \ls A$. For simplicity, we denote $L^p(\R^2)$, and $\calC^{0,\beta}(\R^2)$ as $L^p$, and $\calC^{0,\beta}$, respectively.   

\subsubsection*{Outline of the paper} The rest of this paper is organized as follows. In Section \ref{sec_pre}, we present a crucial estimate on the log-Lipschitz continuity of the velocity field $u$. Sections \ref{sec_well} and \ref{sec_ill} are devoted to providing the details of proofs of Theorems \ref{thm_main} and \ref{thm:inf}, respectively.

  %%%%%%%%%%%%%%%%%%%%%%%%%%%%%%%%%%%%%%%%%%%%%%%
%
%
%		Section 4: Extension to general kernels
%
%
%
%%%%%%%%%%%%%%%%%%%%%%%%%%%%%%%%%%%%%%%%%%%%%%% 
\section{Lipschitz estimates}\label{sec_pre}
\setcounter{equation}{0}
In this section, we provide Lipschitz estimates of the velocity field. Specifically, we show that $u$ is log-Lipschitz continuous when $\theta \in \calC^{0,\alpha}$ and Lipschitz continuous when $\theta \in \calC^{0,\beta}$ for any $\beta \in (\alpha,1]$.
Note that $K_\alpha$ satisfies 
\bq\label{K_alp}
|K_\alpha(x)| \le \frac{C}{|x|^{1+\alpha}} \quad \mbox{and} \quad |\nabla K_\alpha(x)| \le \frac{C}{|x|^{2+\alpha}}, \quad \forall \,x \in \R^2\setminus\{0\}.
\eq
\begin{lemma}\label{log_lip_gen}
Let $h \in \calC^{0,\alpha}(\R^2)$ and $K_\alpha$ satisfy \eqref{K_alp}. Then $K_\alpha \star h$ is log-Lipschitz continuous. Moreover, we have
\[
|(K_\alpha \star h) (x) - (K_\alpha \star h) (\tilde{x})| \le C\|h\|_{\calC^{0,\alpha}}|x-\tilde{x}|(1-\log^-|x-\tilde{x}|),
\]
where $C > 0$ is independent of $x$ and $\tilde{x}$.
\end{lemma}
\begin{proof}
We split the proof into two cases  $|x-\tilde{x}|\ge 1$ and  $|x-\tilde{x}| \le 1$. 

\medskip

\paragraph{\bf (Case A: $|x-\tilde{x}|\ge 1$)} In this case, we estimate
\[\begin{aligned}
(K_\alpha \star h)(x) - (K_\alpha \star h)(\tilde{x}) &=\lt( \int_{ \{ |x-y|\le 2|x-\tilde{x}| \}} + \int_{\{2|x-\tilde{x}|<|x-y|\}} \rt)(K_\alpha(x-y) - K_\alpha(\tilde{x}-y))h(y)\,dy \\
&=: I +II,
\end{aligned}\]
where
\[\begin{aligned}
I&\le \|h\|_{L^\infty}\int_{\{|x-y|\le 2|x-\tilde{x}|\}} \lt(|K_\alpha(x-y)|+ |K_\alpha(\tilde{x}-y)|\rt)\,dy \\
&\ls \|h\|_{L^\infty} \lt(\int_0^{2|x-\tilde{x}|} \frac1{s^\alpha}\,ds + \int_0^{3|x-\tilde{x}|} \frac1{s^\alpha}\,ds \rt)\\
&\ls \|h\|_{L^\infty} |x-\tilde{x}|^{1-\alpha}\cr
& \ls \|h\|_{L^\infty}|x-\tilde{x}|.
\end{aligned}\]
Here we used $|x-\tilde{x}|\ge 1$ and $\{|(x-\tilde{x})-y|\le 2|x-\tilde{x}|\} \subseteq \{|y|\le 3|x-\tilde{x}|\}$.

For $II$, note that $2|x-\tilde{x}| < |x-y|$ implies 
\[
|\tilde{x}-y| = |(\tilde{x}-x) + x-y| \ge \frac{|x-y|}{2},
\] 
and thus $K_\alpha$ is $C^1$ in the region $\{2|x-\tilde{x}|<|x-y|\}$. This deduces
\[\begin{aligned}
II &= \int_{\{2|x-\tilde{x}|<|x-y|\}} \int_0^1 \nabla K_\alpha ((\tilde{x}-y) + t(x-\tilde{x})) \cdot (x- \tilde{x}) h(y)\,dtdy\\
&\ls \|h\|_{L^\infty} |x-\tilde{x}| \int_{\{2|x- \tilde{x}|<|x-y|\}} \frac{1}{|x-y|^{2+\alpha}}\,dy\\
&\ls \|h\|_{L^\infty} |x-\tilde{x}| \int_{2|x-\tilde{x}|}^\infty  \frac1{s^{\alpha+1}}\,ds\\
&\ls \|h\|_{L^\infty} |x-\tilde{x}|  \int_{2}^\infty  \frac1{s^{\alpha+1}}\,ds\cr
&\simeq \|h\|_{L^\infty} |x-\tilde{x}|.
\end{aligned}\]

\medskip

\paragraph{\bf (Case B: $|x-\tilde{x}|<1$)} For simplicity, we write $r := |x-\tilde{x}|$. Then, we estimate the difference as follows:
\[\begin{aligned}
(&K_\alpha \star h)(x) - (K_\alpha \star h)(\tilde{x})\\
&= \frac12\lt(\int_{\{|x-y|\ge 2\}} + \int_{\{ 2r<|x-y|<2 \}} + \int_{\{ |x-y|\le 2r\}} \rt)\lt( K_\alpha(x-y) -K_\alpha(\tilde{x}-y)\rt)h(y)\,dy\\
&\quad + \frac12 \lt(\int_{\{|\tilde{x}-y|\ge 2\}} + \int_{\{ 2r<|\tilde{x}-y|<2 \}} + \int_{\{ |\tilde{x}-y|\le 2r \}} \rt)\lt( K_\alpha(x-y) - K_\alpha(\tilde{x}-y)\rt)h(y)\,dy\\
&=: III_1 + III_2 + III_3 + IV_1 + IV_2 + IV_3. 
\end{aligned}\]
For $III_1$ and $IV_1$, we use the fact $|x - \tilde{x}|\le 1$ to observe
\[
|x-y|, \ |\tilde{x}-y|\ge 1 \quad \mbox{if }  |x-y|\ge 2 \mbox{ or } |\tilde{x}-y|\ge 2. 
\]
Thus, $K_\alpha$ is $\calC^1$ on regions $\{|x-y|\ge 2\}$ and $\{|\tilde{x}-y|\ge 2\}$. Hence, we obtain
\[\begin{aligned}
III_1 &= \int_{\{|x-y|>2\}} \int_0^1 \nabla K_\alpha((\tilde{x}-y) + t(x-\tilde{x})) \cdot(x -\tilde{x}) h(y)\,dtdy\\
&\ls |x-\tilde{x}| \int_{\{|x-y|>2\}} \frac{1}{|(\tilde{x}-y)+t(x-\tilde{x})|^{2+\alpha}}|h(y)|\,dy\\
&\ls \|h\|_{L^\infty}|x-\tilde{x}|,
\end{aligned}\]
and we also have the same estimates for $IV_1$.

For $III_2$, $K_\alpha$ is also in $\calC^1$ in the region $\{2r < |x-y|<2\}$. Thus, we use the mean-value theorem to yield
\begin{align*}
III_2 &=  \int_{\{ 2r<|x-y|<2 \}} \int_0^1 \nabla K_\alpha((\tilde{x}-y) + t(x - \tilde{x}))\cdot (x-\tilde{x}) h(y)\,dtdy\\
&=   \int_{\{ 2r<|x-y|<2 \}} \int_0^1 \nabla K_\alpha((\tilde{x}-y) + t(x - \tilde{x}))\cdot (x-\tilde{x}) (h(y)-h(\tilde{x} + t(x-\tilde{x}))\,dtdy\\
&\quad + \int_{\{ 2r<|x-y|<2 \}} \int_0^1\nabla K_\alpha((\tilde{x}-y) + t(x - \tilde{x})) \cdot (x-\tilde{x}) (h(\tilde{x} + t(x-\tilde{x})) - h(\tilde{x}))\,dtdy\\
&\quad  + \int_{\{ 2r<|x-y|<2 \}} \int_0^1 \nabla K_\alpha((\tilde{x}-y) + t(x - \tilde{x})) \cdot (x-\tilde{x}) h(\tilde{x})\,dtdy\\
&\le C[h]_{\calC^{0,\alpha}}\int_{\{ 2r<|x-y|<2 \}} \frac{|x-\tilde{x}|}{|x-y|^2}h(y)\,dy + C[h]_{\calC^{0,\alpha}}\int_{\{ 2r<|x-y|<2 \}} \frac{|x-\tilde{x}|^{1+\alpha}}{|x-y|^{2+\alpha}}h(y)\,dy\\
&\quad + \int_{\{ 2r<|x-y|<2 \}} \lt( K_\alpha(x-y) - K_\alpha(\tilde{x}-y)\rt) h(\tilde{x})\,dy\\
&\le C[h]_{\calC^{0,\alpha}} |x-\tilde{x}|   \lt(  \int_{2r}^2 \frac 1s\,ds  + |x-\tilde{x}|^{\alpha} \int_{2r}^2 \frac{1}{s^{1+\alpha}}\,ds\rt)\\
&\quad + \int_{\{ 2r<|x-y|<2 \}} \lt(K_\alpha(x-y)-K_\alpha(\tilde{x}-y)\rt) h(\tilde{x})\,dy\\
&\le -C[h]_{\calC^{0,\alpha}} |x- \tilde{x}|\log|x-\tilde{x}|  + C[h]_{\calC^{0,\alpha}} |x- \tilde{x}|  \cr
&\quad + h(\tilde{x}) \int_{\{ 2r<|x-y|<2 \}} \lt(K_\alpha(x-y) - K_\alpha(\tilde{x}-y)\rt) dy,
\end{align*}
where we used 
\[
|(\tilde{x}-y) + t(x - \tilde{x})| = |(x-y) -(1-t)(x-\tilde{x})| \ge | |x-y| - |(1-t)|x-\tilde{x}|| \ge \frac{|x-y|}{2}. 
\]
The same arguments apply to $IV_2$, and we arrive at
\[
IV_2 \le -C[h]_{\calC^{0,\alpha}}|x- \tilde{x}|\log|x-\tilde{x}|+ C[h]_{\calC^{0,\alpha}} |x- \tilde{x}|    + h(\tilde{x})\int_{\{ 2r<|\tilde{x}-y|<2 \}} \lt( K_\alpha(x-y)-K_\alpha(\tilde{x}-y)\rt) dy.
\]
Here, we notice that
\[
\int_{\{2r < |x - y|<2  \}} K_\alpha(x-y)\,dy = \int_{\{2r < |\tilde{x} - y|<2  \}} K_\alpha(\tilde{x}-y)\,dy =  \int_{\{2r < |y|<2  \}}K_\alpha(y)\,dy,
\]
and
\[
\int_{\{2r < |\tilde{x} - y|<2  \}} K_\alpha(x-y)\,dy = \int_{\{2r < |x - y|<2  \}} K_\alpha(\tilde{x}-y)\,dy =  \int_{\{2r < |y|<2  \}} K_\alpha((x-\tilde{x})-y)\,dy.
\]
This shows that the terms with $h(\tilde{x})$ on the right-hand sides of the last inequalities in $III_2$ and $IV_2$ cancel out each other. Thus we have
\[
III_2 + IV_2 \le -C[h]_{\calC^{0,\alpha}}|x- \tilde{x}|\log|x-\tilde{x}| + C[h]_{\calC^{0,\alpha}} |x- \tilde{x}| \leq C[h]_{\calC^{0,\alpha}}|x-\tilde{x}|(1-\log^-|x-\tilde{x}|).
\]
\noindent For $III_3$, we use $\{ |x -y| < 2r\} \subseteq \{ |\tilde{x} -y| < 3r\}$ to get
\[\begin{aligned}
III_3 &= \int_{\{|x -y|< 2r \}} K_\alpha(x-y)(h(y) - h(x))\,dy + h(x) \int_{\{|x -y|< 2r \}} K_\alpha(x-y)\,dy\\
&\quad - \int_{\{|x -y|< 2r \}} K_\alpha(\tilde{x}-y)(h(y) - h(\tilde{x}))\,dy - h(\tilde{x})\int_{\{|x -y|< 2r \}} K_\alpha(\tilde{x}-y)\,dy\\
&\le C[h]_{\calC^{0,\alpha}}\lt(\int_0^{2r} \,ds + \int_0^{3r}\,ds\rt) + h(x) \int_{\{|x -y|< 2r \}} K_\alpha(x-y)\,dy  - h(\tilde{x})\int_{\{|x -y|< 2r \}} K_\alpha(\tilde{x}-y)\,dy\\
&\le C[h]_{\calC^{0,\alpha}} |x - \tilde{x}| + h(x) \int_{\{|y|< 2r \}}K_\alpha(y)\,dy - h(\tilde{x})\int_{\{|y|< 2r \}} K_\alpha((x-\tilde{x})-y)\,dy.
\end{aligned}\]
Similarly, we find
\[\begin{aligned}
IV_3 &\le C[h]_{\calC^{0,\alpha}} |x - \tilde{x}| + h(x) \int_{\{|y|< 2r \}}K_\alpha((x-\tilde{x})-y)\,dy- h(\tilde{x})\int_{\{|y|< 2r \}} K_\alpha(y)\,dy.
\end{aligned}\]
Thus, once we use $\{|(x - \tilde{x})-y| \le 2r \} \subseteq  \{ |y| \le 3r\}$ to deduce
\[\begin{aligned}
III_3 + IV_3 &\ls [h]_{\calC^{0,\alpha}}|x - \tilde{x}| +  |h(x) -h(\tilde{x})|\int_{\{ |y| \le 2r\}} \frac{1}{|y|^{1+\alpha}}\,dy + |h(x) - h(\tilde{x})| \int_{\{ |y|\le 3r\}}\frac{1}{|y|^{1+\alpha}}\,dy\\
&\ls [h]_{\calC^{0,\alpha}}|x - \tilde{x}|  +  [h]_{\calC^{0,\alpha}}|x - \tilde{x}|^{\alpha} \lt( \int_0^{2r} \frac1{s^\alpha}\,ds + \int_0^{3r}\frac1{s^\alpha}\,ds\rt)\\
&\ls [h]_{\calC^{0,\alpha}}|x - \tilde{x}|.
\end{aligned}\]
Therefore, we combine all the above estimates to yield the desired result.
\end{proof}

Following almost the same argument as in the proof of Lemma \ref{log_lip_gen}, we also have the following result.
\begin{corollary}\label{lip_gen}
Let $h \in \calC^{0,\beta}(\R^2)$ with $\alpha<\beta \le 1$ and $K_\alpha$ satisfy \eqref{K_alp}. Then $K_\alpha \star h$ is Lipschitz continuous. Moreover, we have
\[
|(K_\alpha \star h) (x) - (K_\alpha \star h) (\tilde{x})| \le C\|h\|_{\calC^{0,\beta}}|x-\tilde{x}|,
\]
where $C > 0$ is independent of $x$ and $\tilde{x}$.
\end{corollary}

%%%%%%%%%%%%%%%%%%%%%%%%%%%%%%%%%%%%%%%%%%%%%%%
%
%
%
%
%
%
%%%%%%%%%%%%%%%%%%%%%%%%%%%%%%%%%%%%%%%%%%%%%%%
\section{Well-posedness in $\calC^{0,\beta}(\R^2)$ with $\beta \in (\alpha,1)$}\label{sec_well}
In this section, we study the local well-posedness of solutions to the system in $\calC^{0,\beta}(\R^2)$ with $\beta \in (\alpha,1)$ giving the details of proof of Theorem \ref{thm_main}. The proof consists of three parts: existence, uniqueness, and stability.

%%%%%%%%%%%%%%%%%%%%%%%%%%%%%%%%%%%%%%%%%%%%%%%
%
%
%
%
%
%
%%%%%%%%%%%%%%%%%%%%%%%%%%%%%%%%%%%%%%%%%%%%%%%
\subsection{Existence}\label{ext_holder}
In order to show the existence of H\"older continuous solutions to the system \eqref{main_eq}, we regularize the velocity field $u$ as
\[
u^\e(t,x) = (K_\alpha^\e \star \theta)(t,x) \quad \mbox{with } K_\alpha^\e(x) := \frac{x^\perp}{(|x|^2 + \e)^{\frac{2+\alpha}2}}
\]
and consider the following regularized equation:
\bq\label{reg_eq}
\pa_t \theta^\e + u^\e \cdot \nabla \theta^\e = 0
\eq
subject to the regularized initial data:
\[
\theta^\e_0 := \theta_0 \star \phi_\e,
\]
where $\phi_\e$ is the standard mollifier. For the regularized equation \eqref{reg_eq} with $\e > 0$, we obtain the global-in-time existence and uniqueness of regular solutions by employing the classical well-posedness theory, see \cite{Dob79, Lau07} for instance.
%%%%%%%%%%%%%%%%%%%%%%%%%%%%%%%%%%%%%%%%%%%%%%%
%
%
%
%
%
%
%%%%%%%%%%%%%%%%%%%%%%%%%%%%%%%%%%%%%%%%%%%%%%%
\subsubsection{Uniform-in-$\e$ bound estimates} 
We first show the uniform-in-$\e$ bound estimate of solutions $\theta^\e$ in the H\"older norm $\calC^{0,\beta}(\R^2)$. For this, we consider the backward characteristics $\Phi(s;t,\cdot) : \R^2 \to \R^2$ defined as a solution to the following:
 \[
 \frac {d}{ds} \Phi^\e(s;t,x) = u^\e(s,\Phi^\e(s;t,x)), \qquad \Phi^\e(t;t,x) = x
 \] 
 for $s \in (0,t)$. Due to the regularization, the above characteristic is well-defined. Then, we readily observe
 \[
 \theta^\e(t,x) = \theta^\e_0(\Phi^\e(0;t,x)),
 \]
 and thus
 \[
 \|\theta^\e(t)\|_{L^\infty} = \|\theta^\e_0\|_{L^\infty} \leq \|\theta_0\|_{L^\infty}.
 \]
 We also find from Corollary \ref{lip_gen} that for any $x,y \in \R^2$
\begin{align*}
|\Phi^\e(s;t,x) - \Phi^\e(s;t,y)| &\leq |x-y| + \int_s^t |(K_\alpha^\e \star \theta^\e)(\tau, \Phi^\e(\tau;t,x)) - (K_\alpha^\e \star \theta^\e)(\tau, \Phi^\e(\tau;t,y))|\,d\tau\cr
&\leq |x-y| + C\|\theta^\e\|_{\calC^{0,\beta}} \int_s^t |\Phi^\e(\tau;t,x) - \Phi^\e(\tau;t,y)|\,d\tau.
\end{align*} 
Thus, we deduce
\[
|\Phi^\e(0;t,x) - \Phi^\e(0;t,y)| \leq |x-y| \exp\lt(C \|\theta^\e\|_{L^\infty(0,t;\calC^{0,\beta})} t  \rt)
\]
for some $C > 0$ independent of $\e$ and $t$. Together with this, we estimate
\begin{align*}
|\theta^\e(t,x) -  \theta^\e(t,y)| &= |\theta^\e_0(\Phi^\e(0;t,x)) - \theta^\e_0(\Phi^\e(0;t,y))| \cr
&\leq \|\theta^\e_0\|_{\calC^{0,\beta}} |\Phi^\e(0;t,x) - \Phi^\e(0;t,y)|^\beta\cr
&\leq \|\theta_0\|_{\calC^{0,\beta}} \exp\lt(C\beta \|\theta^\e\|_{L^\infty(0,t;\calC^{0,\beta})} t  \rt)|x-y|^\beta
\end{align*} 
and subsequently,
\[
\|\theta^\e(t)\|_{\calC^{0,\beta}} \leq  \|\theta_0\|_{\calC^{0,\beta}}\exp\lt(C\beta \|\theta^\e\|_{L^\infty(0,t;\calC^{0,\beta})} t  \rt).
\]
We then use the standard continuity argument to conclude that there exists $T > 0$ such that
\[
\sup_{0 \leq t \leq T}\|\theta^\e(t)\|_{\calC^{0,\beta}} \leq 4\|\theta_0\|_{\calC^{0,\beta}}.
\]

Moreover, it follows from Corollary \ref{lip_gen} that $u^\e$ is uniformly bounded in $L^\infty(0,T; W^{1,\infty})$.

\subsubsection{Passing to the limit $\e \to 0$} Since $ \calC^{0,\beta} \hookrightarrow  \hookrightarrow \calC^0$, there is a limit function $\theta \in L^\infty(0,T; L^1 \cap \calC^{0,\beta})$ such that 
\[
\theta^\e \to \theta \quad \mbox{in } L^\infty(0,T; L^p(\R^2)) \quad \mbox{and} \quad u^\e \to u = K_\alpha \star \theta \quad \mbox{in } L^\infty(0,T; L^p(\R^2))
\]
for any $p \in (1,\infty]$. We now show that the limit functions $\theta$ and $u$ satisfy the equation \eqref{main_eq} in the sense of distributions. For any $\varphi \in \calC^\infty_c([0,T] \times \R^2)$, we estimate
\begin{align*}
\int_0^t \intr \nabla \varphi \cdot (\theta^\e u^\e - \theta u)\,dxds &= \int_0^t \intr \nabla \varphi \cdot (\theta^\e u^\e - \theta u)\,dxds\cr
&=\int_0^t \intr \nabla \varphi \cdot \lt( ((K_\alpha^\e - K_\alpha) \star \theta)\theta + K_\alpha^\e \star (\theta^\e - \theta) \theta + u^\e (\theta^\e -\theta)\rt) dxds\cr
&=: I^\e + II^\e + III^\e.
\end{align*} 
For $I^\e$, we obtain
\begin{align*}
I^\e &\leq \|(K_\alpha^\e - K_\alpha) {\bf 1}_{|x| \leq 1}\|_{L^1(\R^2)} \|\theta\|_{L^\infty((0,T)\times \R^2)}\|(\nabla \varphi)\theta\|_{L^1((0,T)\times \R^2)} \cr
&\quad +  \|(K_\alpha^\e - K_\alpha) {\bf 1}_{|x| \geq 1}\|_{L^\frac3{1+\alpha}(\R^2)} \|\theta\|_{L^\infty(0,T; L^{\frac3{2-\alpha}}(\R^2))}\|(\nabla \varphi)\theta\|_{L^1((0,T)\times \R^2)} 
\end{align*} 
and the dominated convergence theorem gives $I^\e \to 0$ as $\e \to 0$. 

We next estimate 
\begin{align*}
II^\e &\leq \lt(\|K_\alpha {\bf 1}_{|x| \leq 1}\|_{L^{p'}}\|\theta^\e - \theta\|_{L^p} +  \|K_\alpha {\bf 1}_{|x| \geq 1}\|_{L^\frac3{1+\alpha}} \|\theta^\e - \theta\|_{L^\infty(0,T; L^{\frac3{2-\alpha}})}\rt)\|\theta\nabla \varphi\|_{L^\infty(0,T;L^1)}\cr
&\leq C\lt(\|\theta^\e - \theta\|_{L^p} +  \|\theta^\e - \theta\|_{L^\infty(0,T; L^{\frac3{2-\alpha}})} \rt)\cr
&\to 0
\end{align*} 
as $\e \to 0$, where $p \in (\frac{2}{1-\alpha}, \infty)$ and $p'$ is the H\"older conjugate of $p$.

We finally show that
\[
III^\e \leq \|u^\e\|_{L^\infty((0,T) \times \R^2)}\|\theta^\e - \theta\|_{L^\infty} \|\nabla \varphi\|_{L^1} \leq C\|\theta^\e - \theta\|_{L^\infty} \to 0 \quad \mbox{as } \e \to 0.
\]
This completes the proof of the existence of H\"older continuous solutions $\theta \in \calC([0,T);\calC^{0,\beta}(\R^2))$ to \eqref{main_eq}.

%%%%%%%%%%%%%%%%%%%%%%%%%%%%%%%%%%%%%%%%%%%%%%%
%
%
%
%
%
%
%%%%%%%%%%%%%%%%%%%%%%%%%%%%%%%%%%%%%%%%%%%%%%%
\subsection{Uniqueness}\label{uniq_holder}
In this subsection, we discuss the uniqueness of solutions to \eqref{main_eq} on $L^1\cap \calC^{0,\beta}$ with $\beta \in (\alpha,1]$. 

For this, we need the following technical lemmas.
\begin{lemma}\cite{Li19}\label{frac_leib}
Let $d \in \N$. For $s>0$, suppose $A^s$ is a differential operator such that its symbol $\widehat {A^s}(\xi)$ is a homogeneous function of degree $s$ and $\widehat{A^s}(\xi) \in \mc^\infty(\mathbb{S}^{d-1})$. Then for any $p\in (1,\infty)$, we have
\[
\| A^s(fg) - fA^s g - gA^s f\|_{L^p} \lesssim \|f\|_{L^p} \|\Lambda^s g\|_{L^\infty}.
\]
\end{lemma}

\begin{lemma}\cite{CJ1}\label{lem_mod}
Let $d \in \N$. For $T>0$, suppose that the pairs $(\bar\rho, \bar u)$ and $(\rho,u)$ satisfy the followings:
\begin{enumerate}
\item[(i)]
$(\bar\rho, \bar u)$ and $(\rho,u)$ satisfy the continuity equations in the sense of distribution:
\[
\pa_t \bar\rho + \nabla \cdot (\bar\rho \bar u) =0 \quad \mbox{and} \quad \pa_t \rho + \nabla \cdot (\rho u) =0,
\]
\item[(ii)]
$(\bar\rho, \bar u)$ and $(\rho,u)$ satisfy the energy inequality:
\[
\sup_{0\le t \le T}\lt( \int_{\R^d} \bar\rho |\bar u|^2\,dx + \int_{\R^d} \bar\rho K\star\bar\rho\,dx\rt) <\infty, \quad \sup_{0\le t \le T}\lt( \int_{\R^d} \rho | u|^2\,dx + \int_{\R^d} \rho K\star\rho\,dx\rt) <\infty,
\]
\item[(iii)]
$\bar\rho, \rho\in \mc(0,T; L^1(\R^d))$, $\nabla u\in L^\infty((0,T) \times \R^d)$.
\end{enumerate}
Then for any $\gamma \in [-1, 0)$, we have
\[%\bq\label{res_thm}
\begin{aligned}
\frac12\frac{d}{dt}\int_{\R^d} (\rho -\bar\rho) \Lambda^{-2\gamma} \star (\rho - \bar\rho)\,dx &\leq \int_{\R^d} \bar\rho(u - \bar u) \cdot \nabla \Lambda^{-2\gamma} (\rho - \bar\rho)\,dx  + C\int_{\R^d} (\rho - \bar\rho) \Lambda^{-2\gamma} (\rho - \bar\rho)\,dx
\end{aligned}
\]%\eq
for $t \in [0,T)$ and some $C>0$ which depends only on $\gamma$, $d$ and $\|\nabla u\|_{L^\infty((0,T) \times \R^d)}$.
\end{lemma}

Now, we assume that we have two solutions $\theta_1$ and $\theta_2$ to \eqref{main_eq} correpsonding to initial data $\theta_{1,0}$ and $\theta_{2,0}$, respectviely. Once we write $u_i = \nabla^\perp \Lambda^{-2+\alpha}\theta_i$ for $i=1,2$, we use Lemmas \ref{frac_leib} and \ref{lem_mod} to obtain
\[\begin{aligned}
\frac12\frac{d}{dt}\|\theta_1 - \theta_2\|_{\dot{H}^{\alpha-1}}^2 &\le C \intr \theta_1 (u_1 - u_2)\cdot \nabla \Lambda^{-2+2\alpha}(\theta_1 - \theta_2)\,dx + C\|\theta_1 - \theta_2\|_{\dot{H}^{\alpha-1}}^2 \\
&=\intr \Big[\Lambda^{-1+\alpha}\nabla \cdot (\theta_1 (u_1 - u_2)) - \Lambda^{-1+\alpha}\nabla \theta_1 \cdot (u_1 -u_2)  \\
&\hspace{4.8cm}- \theta_1 \Lambda^{-1+\alpha}\nabla \cdot (u_1 -u_2) \Big] \Lambda^{-1+\alpha}(\theta_1 - \theta_2)\,dx\\
&\quad +\intr  \Lambda^{-1+\alpha}\nabla \theta_1 \cdot (u_1 -u_2) \ \Lambda^{-1+\alpha}(\theta_1 - \theta_2)\,dx +  C\|\theta_1 - \theta_2\|_{\dot{H}^{\alpha-1}}^2\\
&\le C \|\Lambda^\alpha \theta_1\|_{L^\infty} \|u_1 -u_2\|_{L^2} \|\theta_1 - \theta_2\|_{\dot{H}^{\alpha-1}}+ C\|\Lambda^{-1+\alpha}\nabla \theta_1\|_{L^\infty} \|u_1 -u_2\|_{L^2} \|\theta_1 - \theta_2\|_{\dot{H}^{\alpha-1}} \\
&\quad+ C\|\theta_1 -\theta_2\|_{\dot{H}^{\alpha-1}}^2\\
&\le C\|\theta_1 -\theta_2\|_{\dot{H}^{\alpha-1}}^2,
\end{aligned}\]
 where we used $\|u_1 - u_2\|_{L^2} = \|\nabla^\perp \Lambda^{-2+\alpha}(\theta_1 -\theta_2)\|_{L^2} \lesssim \|\theta_1 -\theta_2\|_{\dot{H}^{\alpha-1}}$ and
 \[\begin{aligned}
 \Lambda^\alpha \theta_1 &= \intr \frac{\theta_1(x) - \theta_1(y)}{|x-y|^{2+\alpha}}\,dy\\
 &= \lt(\int_{\{|x-y|\le 1\}} + \int_{\{ |x-y|>1\}} \rt)\frac{\theta_1(x) - \theta_1(y)}{|x-y|^{2+\alpha}}\,dy\\
 &\le C[\theta_1]_{\mc^{0,\beta}}\int_{\{|x-y|\le 1\}} \frac{1}{|x-y|^{2+\alpha-\beta}}\,dy + C\|\theta_1\|_{L^\infty} \int_{\{ |x-y|>1\}}\frac{1}{|x-y|^{2+\alpha}}\,dy\\
 &\le C\|\theta_1 \|_{\mc^{0,\beta}}.
 \end{aligned}\]
 We can also apply the same estimates to $\|\Lambda^{-1+\alpha}\nabla\theta_1\|_{L^\infty}$ and hence, we use Gr\"onwall's lemma to yield
 \[
 \|\theta_1 - \theta_2\|_{\dot{H}^{\alpha-1}}^2 \le C\|\theta_{1,0} - \theta_{2,0}\|_{\dot{H}^{\alpha-1}}^2 
 \] 
 and from which we attain the desired uniqueness.
 
 %%%%%%%%%%%%%%%%%%%%%%%%%%%%%%%%%%%%%%%%%%%%%%%
%
%
%
%
%
%
%%%%%%%%%%%%%%%%%%%%%%%%%%%%%%%%%%%%%%%%%%%%%%%

\subsection{Stability in the little H\"older space}\label{sec:stab}

In this part, we provide the stability of solutions to \eqref{main_eq} in little H\"older spaces. First, we define the little H\"older spaces as
\[
c^{k,\gamma}(\R^2):= \lt\{ f \in \mc^{k,\gamma}(\R^2) \ \bigg| \ \lim_{R \to 0}\sup_{|x-y|\le R}\frac{|D^k f(x) - D^k f(y)|}{|x-y|^\gamma} = 0 \rt\},
\]
which is the closure of $\mc^\infty(\R^2)$ in the usual H\"older norm from $\mc^{k,\gamma}(\R^2)$. Now, for given $\theta_0 \in c^{0,\beta}$, $T>0$ and the corresponding unique solution $\theta \in \mc(0,T; c^{0,\beta})$, we show that for any $\e>0$, there exists $\delta>0$ such that whenever $\tilde\theta_0 \in c^{0, \beta}$ satisfies $\|\theta_0 - \tilde\theta_0\|_{\mc^{0,\beta}}<\delta$, the unique solution $\tilde\theta$ corresponding to $\tilde\theta_0$ satisfies
\[
\sup_{t \in [0,T)} \|(\theta- \tilde\theta)(t)\|_{\mc^{0,\beta}} <\e.
\]

For this, consider a standard mollifier $\phi_\nu(x) =: \frac{1}{\nu^2} \phi(\frac{x}{\nu})$ and let $\theta^\nu := \phi_\nu \star \theta$ define $\tilde\theta^\nu$ similarly. Then one has
\[
\|(\theta - \tilde\theta)(t)\|_{\mc^{0,\beta}} \le \|(\theta-\theta^\nu)(t)\|_{\mc^{0,\beta}} + \|(\theta^\nu - \tilde\theta^\nu)(t)\|_{\mc^{0,\beta}} + \|(\tilde\theta^\nu - \tilde\theta)(t)\|_{\mc^{0,\beta}}
\]
Here, we first need to show that $ \sup_{t\in [0,T)}\|(\theta-\theta^\nu)(t)\|_{\mc^{0,\beta}}$ and $\sup_{t\in [0,T)} \|(\tilde\theta^\nu - \tilde\theta)(t)\|_{\mc^{0,\beta}}$ goes to 0 as $\nu$ tends to 0. For this, one easily sees that
\[\begin{aligned}
\|(\theta-\theta^\nu)(t)\|_{L^\infty} &\le \sup_{x \in \R^2}\int_{|y|\le \nu} \phi_\nu(y)|\theta(x-y) - \theta(x)|\,dy\le \|\theta(t)\|_{\mc^{0,\beta}} \nu^\beta.
\end{aligned}\]
Moreover, one observes that
\[\begin{aligned}
&\frac{|(\theta-\theta^\nu)(x) - (\theta-\theta^\nu)(y)|}{|x-y|^\beta}\cr
&\quad \le \frac{|(\theta-\theta^\nu)(x) - (\theta-\theta^\nu)(y)|}{|x-y|^\beta} \mathds{1}_{\{|x-y|< \sqrt \nu\}}  + \frac{|(\theta-\theta^\nu)(x) - (\theta-\theta^\nu)(y)|}{|x-y|^\beta} \mathds{1}_{\{ |x-y| \ge \sqrt\nu\}}\\
&\quad \le \sup_{|x-y|<\sqrt \nu } \frac{|\theta(x) - \theta(y)|}{|x-y|^\beta} + \sup_{|x-y|<\sqrt \nu } \int_{|z|\le \nu}\phi_\nu(z) \frac{|\theta(x-z)-\theta(y-z)|}{|x-y|^\beta}\,dz + 2\nu^{\frac\beta2}\|\theta(t)\|_{\mc^{0,\beta}}\\
&\quad \le 2 \sup_{|x-y|<\sqrt \nu } \frac{|\theta(x) - \theta(y)|}{|x-y|^\beta} + 2\nu^{\frac\beta2}\|\theta(t)\|_{\mc^{0,\beta}}\\
&\quad \to 0
\end{aligned}\]
as $\nu$ tends to 0. Thus, we can choose $\nu_1 = \nu_1(\e)$ such that
\[
 \sup_{t\in [0,T)}\|(\theta-\theta^{\nu_1})(t)\|_{\mc^{0,\beta}}<\frac\e3 \quad  \mbox{and}\quad \sup_{t\in [0,T)} \|(\tilde\theta^{\nu_1} - \tilde\theta)(t)\|_{\mc^{0,\beta}} <\frac\e3.
\]
and we note that 
\[
\|f\|_{\mc^{0,\beta}} \lesssim \||\xi|^\beta \hat f\|_{L_\xi^1} \lesssim \|f\|_{\dot{H}^{2	}}^{\frac{2 + \beta-\alpha}{3 -\alpha}}\|f\|_{\dot{H}^{\alpha-1}}^{\frac{1-\beta}{3-\alpha}},
\]
i.e. $\dot{H}^{2} \cap \dot{H}^{\alpha-1} \hookrightarrow \mc^{0,\beta}$. Thus we get
\[\begin{aligned}
\|(\theta - \tilde\theta)(t)\|_{\mc^{0,\beta}} &\le \frac{2}{3}\e +C \|(\theta^{\nu_1} - \tilde\theta^{\nu_1})(t)\|_{\dot{H}^{2}}^{\frac{2 + \beta-\alpha}{3 -\alpha}}\|(\theta^{\nu_1} - \tilde\theta^{\nu_1})(t)\|_{\dot{H}^{\alpha-1}}^{\frac{1-\beta}{3-\alpha}},
\end{aligned}\]
where $C > 0$ is independent of $\nu_1$, and we note that
\bq\label{moll_neg}
\|\phi_\e \star f\|_{\dot{H}^{\alpha-1}} \le \|f\|_{\dot{H}^{\alpha-1}}.
\eq
Indeed, choose any $\psi \in \dot{H}^{1-\alpha}$ with $\|\psi\|_{\dot{H}^{1-\alpha}} \le 1$. Then we get
\[\begin{aligned}
\lt|\intr \psi \phi_\e \star f \,dx\rt| &= \lt|\intr f \phi_\e \star \psi\,dx\rt| \le \|f\|_{\dot{H}^{\alpha-1}}\|\phi_\e \star \psi\|_{\dot{H}^{1-\alpha}},
\end{aligned}\]
and since the change of variables $(x', y') := (x-w, y-w)$ gives
\[\begin{aligned}
\|\phi_\e \star \psi\|_{\dot{H}^{1-\alpha}}^2 &= \iint_{\R^2 \times \R^2}\frac{|\phi_\e\star\psi (x) -\phi_\e\star\psi(y)|^2}{|x-y|^{2+2(1-\alpha)}}\,dxdy\\
&\le \iint_{\R^2\times \R^2} \int_{|w|\le \e} \frac{\phi_\e(w) |\psi(x-w) - \psi(y-w)|^2}{|x-y|^{2+2(1-\alpha)}}\,dwdxdy\\
&= \iint_{\R^2\times \R^2} \int_{|w|\le \e} \frac{\phi_\e(w) |\psi(x') - \psi(y')|^2}{|x'-y'|^{2+2(1-\alpha)}}\,dwdx'dy' \cr
&= \|\psi\|_{\dot{H}^{1-\alpha}}^2,
\end{aligned}\]
 we arrive at \eqref{moll_neg}. Hence, one has
 \[
\|(\theta^{\nu_1} - \tilde\theta^{\nu_1})(t)\|_{\dot{H}^{\alpha-1}} \le \|(\theta - \tilde\theta)(t)\|_{\dot{H}^{\alpha-1}},
 \]
 and from the uniqueness result, we get
\[
\|(\theta - \tilde\theta)(t)\|_{\dot{H}^{\alpha-1}} \le C \|\theta_0 - \tilde\theta_0\|_{\dot{H}^{\alpha-1}},
\]
where $C > 0$ only depends on $\|\theta\|_{\mc(0,T;\mc^{0,\beta})}$, $\|\tilde\theta\|_{\mc(0,T;\mc^{0,\beta})}$,$T$, $\alpha$ and $\beta$. Thus, we obtain
\[\begin{aligned}
\|(\theta - \tilde\theta)(t)\|_{\mc^{0,\beta}} &\le \frac{2\e}{3} +C_1\|\theta_0 - \tilde\theta_0\|_{\dot{H}^{\alpha-1}}^{\frac{1-\beta}{\frac d2 +2-\alpha}},
\end{aligned}\]
where $C_1 > 0$ depends on $\nu_1$, $\|\theta\|_{\mc(0,T;\mc^{0,\beta})}$, $\|\tilde\theta\|_{\mc(0,T;\mc^{0,\beta})}$, $T$, $\alpha$ and $\beta$. We remark that $C_1$ may grow to infinity as $\e$ tends to 0. Thus, once we choose $\delta$ as
\[
\delta =: \lt( \frac{\e}{3C_1}\rt)^{\frac{\frac d2 + 2-\alpha}{1-\beta}},
\]
one has
\[
\sup_{t \in [0,T)}\|(\theta - \tilde\theta)(t)\|_{\mc^{0,\beta}} < \e,
\]
and this implies our desired stability result.
%%%%%%%%%%%%%%%%%%%%%%%%%%%%%%%%%%%%%%%%
%
%
%
%
%
%
%%%%%%%%%%%%%%%%%%%%%%%%%%%%%%%%%%%%%%%%
\section{Strong ill-posedness in $\calC^{0,\alpha}(\R^2)$}\label{sec_ill}

In this section, we provide the details of the proof for Theorem \ref{thm:inf}.

%we prove Theorem~\ref{thm:inf}. Using Lagrangian approach, one can show for $\beta \in [0,1]$ that $$\sup_{x \neq y} \frac {|\theta(t,x) - \theta(t,y)|}{|x-y|^{\beta}} \leq \left( \sup_{x \neq y} \frac {|\theta_0(x) - \theta_0(y)|}{|x-y|^{\beta}} \right) \exp \left(\beta \int_0^t \| \nabla u(s) \|_{L^{\infty}} \,\mathrm{d}s \right).$$ This a priori estimate with Corollary~\ref{lip_gen} implies that $\calC^{0,\beta}$ regularity is propagated at least for short time when $\beta \in (\alpha,1]$. On the other hand, in the case of $\beta = \alpha$, we do not have Lipschitz estimate of the velocity field, but Lemma~\ref{log_lip_gen}. Note that this log-Lipschitz estimate is not enough to prevent solutions from generating the $\calC^{0,\alpha}$--norm inflation.
%
%To prove that $\alpha$-SQG equation \eqref{main_eq} with $\alpha \in (0,1)$ is strongly ill-posed in the critical space $\calC^{0,\alpha}(\R^2)$, we introduce in the next subsection a family of odd-odd symmetric initial data $\{ \theta_0^{(N)} \}_{N \geq n_0} \subset \calC^{0,\alpha}_c(\mathbb{R}^2)$ that yield the norm inflation of solutions. In Section~\ref{sec_app}, a velocity approximation (Lemma~\ref{key_lem}) is provided with its proof. Then, we study short time dynamics of solutions and complete the proof of Theorem~\ref{thm:inf} in Section~\ref{sec_pf}.

\subsection{Setup of initial data}\label{sec_ini}
We fix $n_0 \geq 1$ and some smooth radial bump function $\phi: \R^2\rightarrow \R$ satisfying 
\[
{\bf 1}_{B(0, \frac1{64})} \leq \phi \leq {\bf 1}_{B(0, \frac1{32})}.
\]
%
%
%the following properties: 
%\begin{itemize}
%	\item $\phi$ is $C^\infty$-smooth and radial. 
%	\item $\phi$ is supported in $B_{0}(\frac{1}{32})$ and $\phi = 1$ in $B_{0}(\frac{1}{64})$. 
%\end{itemize} 
Then for any $N \geq n_0$, we define
\begin{equation}\label{eq:nonexist-data}
	\begin{split}
		\theta_0^{(N)} :=  \sum_{n=n_{0}}^{N}  n^{-\beta} \theta^{(n)}_{0, loc}
	\end{split}
\end{equation}
for some $0 < \beta < 1/4$, where
\begin{equation*}
	\begin{split}
		\theta^{(n)}_{0, loc}(x) := 4^{-\alpha n} \phi( 4^n(x_1-4^{-n-2}, x_2-4^{-n-2} ) )
	\end{split}
\end{equation*}
for $x=(x_1, x_2) \in \R^2$. The specific value of $\beta$ slightly greater than $0$ will be determined later. Next, we extend each of $\theta^{(n)}_{0, loc}$ (and similarly $\theta_{0}^{(N)}$) to $\R^2$ as an odd function with respect to both axes. For any given $\varepsilon>0$, one can have by taking $n_{0}\ge 1$ large enough
\begin{equation}\label{small_ini}
\| \theta_0^{(N)} \|_{\calC^{0,\alpha}} <\varepsilon 
\end{equation}
for all $N \geq n_0$.

\subsection{Velocity approximation for odd-odd solutions}\label{sec_app}
In this subsection, we shall state the velocity estimates for odd-odd solutions from which we have the norm inflation of the family of solutions with the prescribed initial data. For convenience, we shall normalize the SQG Biot--Savart law in a way that  
\begin{equation*}
	u(t,x) =  \int_{\R^2} \frac {(x -y)^{\perp}}{|x - y|^{2+\alpha}} \theta(t,y) \,dy.
\end{equation*} 
\begin{lemma}\label{key_lem}
Let $\theta \in \calC^{0,\alpha}_c$ be an odd function with respect to both axes, i.e., $\theta(x) = - \theta(\bar{x}) = \theta(-x) = -\theta(\tilde{x})$ where $\bar{x} := (x_1,-x_2)$ and $\tilde{x} := (-x_1,x_2)$.
%	We impose the following assumptions on $\theta \in \calC^{0,\alpha}_c$: \begin{itemize}
%		\item $\theta$ is odd with respect to both axes, i.e., $\theta(x) = - \theta(\bar{x}) = \theta(-x) = -\theta(\tilde{x})$ where $\bar{x} := (x_1,-x_2)$ and $\tilde{x} := (-x_1,x_2)$.
%%		\item There exists a constant $R>0$ such that $\theta$ vanishes for all $x \not\in B(0;R)$. 
%	\end{itemize} 
	Then, for any $x$ satisfying $x_1$, $x_2>0$ and $|x| < \frac14$, we have 
	\begin{equation}\label{u1}
	\begin{gathered}
		\left|\frac {u_1(x)}{x_1} - 4(2+\alpha)\int_{Q(x)} \frac { y_1 y_2 }{|y|^{4+\alpha}} \theta(y) \,dy\right| \le C\| \theta \|_{\calC^{0,\alpha}} \left( 1+ \log \frac {x_1^2 + x_2^2}{x_1^2} \right)
	\end{gathered}
	\end{equation} and
	\begin{equation}\label{u2}
	\begin{gathered}
		\left|\frac {u_2(x)}{x_2} + 4(2+\alpha)\int_{Q(x)} \frac { y_1 y_2 }{|y|^{4+\alpha}} \theta(y) \,dy  \right| \le C\| \theta \|_{\calC^{0,\alpha}} \left( 1+ \log \frac {x_1^2 + x_2^2}{x_2^2} \right),
	\end{gathered}
	\end{equation} where $Q(x) := [2x_1,\infty) \times [2x_2,\infty)$.
\end{lemma}
\begin{remark}
	A similar result can be obtained for the periodic case. The error bounds for \eqref{u1} and \eqref{u2} are an improvement of that of \cite[Lemma~A.1]{JK24}. Note that $H^{1+\alpha}(\R^2) \hookrightarrow \calC^{0,\alpha}(\R^2)$ for any $\alpha \in (0,1)$.
\end{remark}
\begin{remark}\label{rmk_R}
We may replace $Q(x)$ by $R(x) := [2x_1,\infty) \times [0,\infty)$ in \eqref{u1} since 
\begin{equation*}
	\begin{aligned}
		\left| \int_{R(x) \setminus Q(x)} \frac { y_1 y_2 }{|y|^{4+\alpha}} \theta(y) \,dy \right| &\leq \| \theta \|_{\calC^{0,\alpha}} \int_{0}^{2x_2} \int_{2x_1}^{\infty} \frac {y_1y_2 }{|y|^{4}} \,dy_1 dy_2 \\
		&\leq C\| \theta \|_{\calC^{0,\alpha}} \int_{0}^{2x_2} \frac {y_2 }{4x_1^2+y_2^2} \,dy_2 \cr
		&\leq C\| \theta \|_{\calC^{0,\alpha}} \log \frac {x_1^2 + x_2^2}{x_1^2} .
	\end{aligned}
\end{equation*}
\end{remark}
\begin{proof}[Proof of Lemma \ref{key_lem}]
Our main idea of the proof is motivated by \cite[Lemma A.1]{JK24} and \cite[Lemma 2.1]{Z15}. Let $x$ and $\theta$ satisfy all assumptions for Lemma~\ref{key_lem}. Then, we have the Biot--Savart law that
	\begin{equation*}
	\begin{gathered}
		u(x) =\int_{[0,\infty)^2} \left( \frac {(x -y)^{\perp}}{|x - y|^{2+\alpha}} - \frac {(x - \tilde{y})^{\perp}}{|x - \tilde{y}|^{2+\alpha}} + \frac {(x + y)^{\perp}}{|x + y|^{2+\alpha}} - \frac {(x - \bar{y})^{\perp}}{|x - \bar{y}|^{2+\alpha}} \right) (\theta(y)-\theta(x)) \,dy \\
		+ \theta(x) \int_{[0,\infty)^2} \frac {(x -y)^{\perp}}{|x - y|^{2+\alpha}} - \frac {(x - \tilde{y})^{\perp}}{|x - \tilde{y}|^{2+\alpha}} + \frac {(x + y)^{\perp}}{|x + y|^{2+\alpha}} - \frac {(x - \bar{y})^{\perp}}{|x - \bar{y}|^{2+\alpha}} \,dy.
	\end{gathered}
\end{equation*} Let $u_1(x) =  I_1 + I_2 + II_1 + II_2$, where
	\begin{equation*}
		\begin{aligned}
			I_1 &:= -\int_{[0,\infty)^2} \left( \frac {x_2 -y_2}{|x - y|^{2+\alpha}} - \frac {x_2 -y_2}{|x - \tilde{y}|^{2+\alpha}} \right) (\theta(y) -\theta(x)) \,dy, \\
			I_2 &:= -\int_{[0,\infty)^2} \left( \frac {x_2 + y_2}{|x + y|^{2+\alpha}} - \frac {x_2 + y_2}{|x - \bar{y}|^{2+\alpha}} \right) (\theta(y) -\theta(x)) \,dy, \\
			II_1 &= - \theta(x) \int_{[0,\infty)^2} \frac {x_2-y_2}{|x - y|^{2+\alpha}} - \frac {x_2-y_2}{|x - \tilde{y}|^{2+\alpha}} \,dy, \quad \mbox{and} \\
			II_2 &= - \theta(x) \int_{[0,\infty)^2} \frac {x_2+y_2}{|x + y|^{2+\alpha}} - \frac {x_2+y_2}{|x - \bar{y}|^{2+\alpha}} \,dy.
		\end{aligned}
	\end{equation*} 
	We first estimate $I_1$. To show that 
	\bq\label{I1_est}
	\left| I_1 - 2(2+\alpha)x_1 \int_{Q(x)} \frac {y_1y_2}{|y|^{4+\alpha}} \theta(y) \,dy \right| \leq Cx_1 \| \theta \|_{\calC^{0,\alpha}} \left( 1+ \log \frac {x_1^2 + x_2^2}{x_1^2} \right),
	\eq
	 we split the domain into three parts $\R^2 = Q(x) \cup \left( [0,2x_1] \times [0,\infty) \right)\cup \left([2x_1,\infty) \times [0,2x_2]\right)$ and estimate the integral on each region. 	\medskip

\paragraph{\bf (i) $y \in Q(x)$ case}	We claim that
	\begin{equation}\label{I1_claim}
	\begin{aligned}
		\left| -\int_{Q(x)} \left( \frac {x_2 -y_2}{|x - y|^{2+\alpha}} - \frac {x_2 -y_2}{|x - \tilde{y}|^{2+\alpha}} \right) (\theta(y) -\theta(x)) \,dy - 2(2+\alpha)x_1 \int_{Q(x)} \frac {y_1y_2}{|y|^{4+\alpha}} \theta(y) \,dy \right| \leq Cx_1 \| \theta \|_{\calC^{0,\alpha}}.
	\end{aligned}
\end{equation} Note that
	\begin{equation*}
	\begin{aligned}
		&-\int_{Q(x)} \left( \frac {x_2 -y_2}{|x - y|^{2+\alpha}} - \frac {x_2 -y_2}{|x - \tilde{y}|^{2+\alpha}} \right) (\theta(y) -\theta(x)) \,dy \\
		&\quad = -\int_{Q(x)} \left( \frac {x_2 -y_2}{|x - y|^{2+\alpha}} - \frac {x_2 -y_2}{|x - \tilde{y}|^{2+\alpha}} \right) \theta(y) \,dy +\int_{Q(x)} \left( \frac {x_2 -y_2}{|x - y|^{2+\alpha}} - \frac {x_2 -y_2}{|x - \tilde{y}|^{2+\alpha}} \right) \theta(x)\,dy.
	\end{aligned}
\end{equation*} It is not difficult to show that 
\begin{equation}\label{I1_i}
	\begin{aligned}
		\frac {1}{2} |y| \leq |x-y| \leq |y| \qquad \mbox{and} \qquad \frac {1}{2} |y| \leq |x-\tilde{y}| \leq \frac {3}{2}|y|, \qquad y \in Q(x).
	\end{aligned}
\end{equation} We define %$f(\tau) = |x-(\tilde{y} - (1-\tau)(\tilde{y}-y))|^{2+\alpha}$ and 
$g(t) = |y-(\tilde{x} - (1-t)(\tilde{x}-x))|^{2+\alpha}$ for $0 \leq t \leq 1$. Then, we use the mean value theorem to obtain
	\begin{align*}
		|\tilde x - y|^{2+\alpha} - |x - y|^{2+\alpha} %&= f(1) - f(0) \\
% 		&= 2(2+\alpha)y_1(\tau(x_1+y_1) + (1-\tau)(x_1- y_1)) |\tau(x-\tilde{y}) + (1-\tau)(x-y)|^{\alpha} \\
		&= g(1)-g(0) \\
		&= 2(2+\alpha)x_1(\tau(y_1+x_1) + (1-\tau)(y_1- x_1)) |\tau(y-\tilde{x}) + (1-\tau)(y-x)|^{\alpha}
	\end{align*} for some $\tau \in (0,1)$. Thus, it follows from the above and the fact $|\tilde x - y| = |x - \tilde y|$ that 
	\begin{equation}\label{fg_est}
	\begin{aligned}
		\frac {x_2 -y_2}{|x - y|^{2+\alpha}} - \frac {x_2 -y_2}{|x - \tilde{y}|^{2+\alpha}} %= \frac {(x_2 -y_2)(f(1)-f(0))}{|x - y|^{2+\alpha}|x - \tilde{y}|^{2+\alpha}} 
		= \frac {(x_2 -y_2)(g(1)-g(0))}{|x - y|^{2+\alpha}|x - \tilde{y}|^{2+\alpha}}.
	\end{aligned}
\end{equation}
	Hence, we combine this with \eqref{I1_i} that
	\begin{equation}\label{II_1}
	\begin{aligned}
		\left| \int_{Q(x)} \left( \frac {x_2 -y_2}{|x - y|^{2+\alpha}} - \frac {x_2 -y_2}{|x - \tilde{y}|^{2+\alpha}} \right) \theta(x)\,dy \right| &\leq |\theta(x)| \int_{Q(x)} \frac {|x_2 -y_2||g(1)-g(0)|}{|x - y|^{2+\alpha}|x - \tilde{y}|^{2+\alpha}} \,dy \\
		&\leq C|\theta(x)| \int_{Q(x)} \frac {x_1}{|y|^{2+\alpha}} \,dy \\
		&\leq Cx_1 \frac{|\theta(x)|}{|x|^\alpha}\cr
		& \le Cx_1\| \theta \|_{\calC^{0,\alpha}},
	\end{aligned}
\end{equation} 
where we used the odd-odd symmetry. Similarly, one can show that
\begin{equation*}
	\begin{aligned}
		\left| -\int_{Q(x)} \left( \frac {x_2}{|x - y|^{2+\alpha}} - \frac {x_2}{|x - \tilde{y}|^{2+\alpha}} \right) \theta(y) \,dy \right| &\leq \int_{Q(x)} \frac {x_2|g(1)-g(0)|}{|x - y|^{2+\alpha}|x - \tilde{y}|^{2+\alpha}} \theta(y) \,dy \\
		&\leq C\int_{Q(x)} \frac {x_1x_2}{|y|^{3}} \frac {\theta(y)}{|y|^{\alpha}}\,dy \\
		&\leq Cx_1 \| \theta \|_{\calC^{0,\alpha}}.
	\end{aligned}
\end{equation*}
Now, it remains to show that
\begin{equation*}
	\begin{aligned}
		\left| \int_{Q(x)} \left( \frac {y_2}{|x - y|^{2+\alpha}} - \frac {y_2}{|x - \tilde{y}|^{2+\alpha}} \right) \theta(y) \,dy - 2(2+\alpha)x_1 \int_{Q(x)} \frac {y_1y_2}{|y|^{4+\alpha}} \theta(y) \,dy \right| \leq Cx_1\| \theta \|_{\calC^{0,\alpha}}.
	\end{aligned}
\end{equation*}
For this, note that
\begin{equation*}
	\begin{aligned}
		&\frac {y_2}{|x - y|^{2+\alpha}} - \frac {y_2}{|x - \tilde{y}|^{2+\alpha}} - 2(2+\alpha)x_1 \frac {y_1y_2}{|y|^{4+\alpha}} \\
		&\quad= 2(2+\alpha)x_1y_2 \left( \frac {(y_1-(1-2\tau)x_1)|\tau(y-\tilde{x}) + (1-\tau)(y-x)|^{\alpha}}{|x - y|^{2+\alpha}|x - \tilde{y}|^{2+\alpha}} - \frac {y_1}{|y|^{4+\alpha}} \right) \\
		&\quad= 2(2+\alpha)x_1y_1y_2 \left( \frac {|y-((1-2\tau)x_1, x_2)|^{\alpha}}{|x - y|^{2+\alpha}|\tilde{x} - y|^{2+\alpha}} - \frac {1}{|y|^{4+\alpha}} \right) \cr
		&\qquad - 2(2+\alpha)(1-2\tau) x_1^2y_2 \frac {|\tau(y-\tilde{x}) + (1-\tau)(y-x)|^{\alpha}}{|x - y|^{2+\alpha}|x - \tilde{y}|^{2+\alpha}}.
	\end{aligned}
\end{equation*} Here, the second term is bounded by 
\[
Cx_1^2\frac {1}{|y|^{3+\alpha}}.
\]
Thus, we arrive at

\[
\left| \int_{Q(x)} 2(2+\alpha)(1-2\tau) x_1^2y_2 \frac {|\tau(y-\tilde{x}) + (1-\tau)(y-x)|^{\alpha}}{|x - y|^{2+\alpha}|x - \tilde{y}|^{2+\alpha}} \theta(y) \,dy \right| \leq Cx_1\| \theta \|_{\calC^{0,\alpha}}.
\]
 On the other hand, we apply the mean value theorem again to 
 
 \[
 h(\eta) = \left| \frac {y}{|y|} - \frac {((1-2\tau)x_1,x_2)}{|y|} - 2\eta \tau \frac {x_1e_1}{|y|} \right|^{\alpha} \left| \frac {y}{|y|} - \eta \frac {\tilde{x}}{|y|} \right|^{\alpha}, \qquad 0 \leq \eta \leq 1,
 \] 
 and yield 
\begin{equation*}
	\begin{aligned}
		&\left| \frac {|y-((1-2\tau)x_1, x_2)|^{\alpha}}{|x - y|^{2+\alpha}|\tilde{x} - y|^{2+\alpha}} - \frac {1}{|y|^{4+\alpha}} \right| \\
		&\quad= \frac {|y|^{4+2\alpha}}{|x - y|^{2+\alpha}|\tilde{x} - y|^{2+\alpha}|y|^{4+\alpha}}  \left| \frac {y}{|y|} - \frac {((1-2\tau)x_1, x_2)}{|y|} \right|^{\alpha} - \left| \frac {y}{|y|} - \frac {x}{|y|} \right|^{\alpha} \left| \frac {y}{|y|} - \frac {\tilde{x}}{|y|} \right|^{\alpha} \\
		&\qquad + \left| \frac {y}{|y|} - \frac {x}{|y|} \right|^{\alpha} \left| \frac {y}{|y|} - \frac {\tilde{x}}{|y|} \right|^{\alpha} - \left( 1 - \frac {2x \cdot y}{|y|^2} + \frac {|x|^2}{|y|^2} \right) \left( 1 - \frac {2\tilde{x} \cdot y}{|y|^2} + \frac{|\tilde{x}|^2}{|y|^2} \right) \left| \frac {y}{|y|} - \frac {x}{|y|} \right|^{\alpha} \left| \frac {y}{|y|} - \frac {\tilde{x}}{|y|} \right|^{\alpha} \Bigg| \\
		&\quad\leq \frac {|y|^{4+2\alpha}}{|x - y|^{2+\alpha}|\tilde{x} - y|^{2+\alpha}|y|^{4+\alpha}} |h(1) - h(0)| + C|x| \frac {1}{|y|^{5+\alpha}} \\
		&\quad\leq C|x| \frac {1}{|y|^{5+\alpha}},
	\end{aligned}
\end{equation*}
where we used
 \begin{equation*}
	\begin{aligned}
		h'(\eta)& = |y|^{-2\alpha} \big( \alpha |y-((1-2\tau)x_1,x_2) -2\eta \tau x_1e_1|^{\alpha-2} (4\eta \tau^2 x_1^2 - 2 \tau x_1 (y_1-(1-2\tau)x_1) ) |y-\eta \tilde{x}|^{\alpha} \\
		&\quad + |y-((1-2\tau)x_1,x_2) -2\eta \tau x_1e_1|^{\alpha} \alpha |y-\eta \tilde{x}|^{\alpha-2} (\eta |\tilde{x}|^2 - y \cdot \tilde{x}) \big) \\
		&\leq C\frac {|x|}{|y|}
	\end{aligned}
\end{equation*} for $y \in Q(x)$. So we attain
\[
\left| \int_{Q(x)} 2(2+\alpha)x_1y_1y_2 \left( \frac {|y-((1-2\tau)x_1, x_2)|^{\alpha}}{|x - y|^{2+\alpha}|\tilde{x} - y|^{2+\alpha}} - \frac {1}{|y|^{4+\alpha}} \right) \theta(y) \,dy \right| \leq Cx_1\| \theta \|_{\calC^{0,\alpha}}.\]
Thus, we combine the above estimates to obtain \eqref{I1_claim}.  	\medskip

\paragraph{\bf (ii) $y \in [0,2x_1] \times [0,\infty)$ case} First, we deduce from \eqref{fg_est} that
		\begin{equation*}
	\begin{gathered}
		-\int_{[0,2x_1] \times [0,\infty)} \left( \frac {x_2 -y_2}{|x - y|^{2+\alpha}} - \frac {x_2 -y_2}{|x - \tilde{y}|^{2+\alpha}} \right) (\theta(y) -\theta(x)) \,dy \\
		= -\int_{[0,2x_1] \times [0,\infty)}  \frac {(x_2 -y_2)(g(1)-g(0))}{|x - y|^{2+\alpha}|x - \tilde{y}|^{2+\alpha}} (\theta(y) - \theta(x)) \,dy.
	\end{gathered}
\end{equation*}
% 
%We show that $$\left| -\int_{[0,2x_1] \times [0,\infty)}  \frac {(x_2 -y_2)(g(1)-g(0))}{|x - y|^{2+\alpha}|x - \tilde{y}|^{2+\alpha}} (\theta(y) - \theta(x)) \,dy \right| \leq C x_1 \|\theta \|_{\calC^{0,\alpha}}.$$ 
Since $|g(1) - g(0)| \leq Cx_1 (x_1+y_1)|\tilde{x}-y|^{\alpha}$ and $x_1+y_1 \leq 3x_1$ in this region, we have
\begin{equation*}
	\begin{aligned}
		&\left| -\int_{[0,2x_1] \times [0,\infty)}  \frac {(x_2 -y_2)(g(1)-g(0))}{|x - y|^{2+\alpha}|x - \tilde{y}|^{2+\alpha}} (\theta(y) - \theta(x)) \,dy \right| \\
		&\quad\leq Cx_1\| \theta \|_{\calC^{0,\alpha}} \int_{0}^{\infty} \frac {x_1}{x_1^2 + (x_2-y_2)^2} \lt(\int_{0}^{x_1} \frac {x_2 -y_2}{(x_1 - y_1)^2 + (x_2 - y_2)^{2}} \,dy_1\rt)dy_2 \\
		&\quad\leq Cx_1\| \theta \|_{\calC^{0,\alpha}} \int_{0}^{\infty} \frac {x_1}{x_1^2 + (x_2-y_2)^2} \,dy_2\cr
		& \quad \leq Cx_1\| \theta \|_{\calC^{0,\alpha}}.
	\end{aligned}
\end{equation*}

\paragraph{\bf (iii) $y \in [2x_1,\infty) \times [0,2x_2]$ case} We again deduce from \eqref{fg_est} that
		\begin{equation*}
	\begin{gathered}
		-\int_{[2x_1,\infty) \times [0,2x_2]} \left( \frac {x_2 -y_2}{|x - y|^{2+\alpha}} - \frac {x_2 -y_2}{|x - \tilde{y}|^{2+\alpha}} \right) (\theta(y) -\theta(x)) \,dy \\
		= -\int_{[2x_1,\infty) \times [0,2x_2]}  \frac {(x_2 -y_2)(g(1)-g(0))}{|x - y|^{2+\alpha}|x - \tilde{y}|^{2+\alpha}} (\theta(y) - \theta(x)) \,dy.
	\end{gathered}
\end{equation*} In this region, we have $|g(1) - g(0)| \leq Cx_1 (x_1+y_1)|\tilde{x}-y|^{\alpha}$ and $y_1 - x_1 \leq y_1 + x_1 \leq 3 (y_1 - x_1)$ for $y_1 \geq 2x_1$. This implies
	\begin{equation*}
	\begin{aligned}
		&\left| -\int_{[2x_1,\infty) \times [0,2x_2]}  \frac {(x_2 -y_2)(g(1)-g(0))}{|x - y|^{2+\alpha}|x - \tilde{y}|^{2+\alpha}} (\theta(y) - \theta(x)) \,dy \right| \\
		&\quad\leq Cx_1\| \theta \|_{\calC^{0,\alpha}} \int_{0}^{x_2} (x_2 -y_2)\lt( \int_{2x_1}^{\infty} \frac {y_1-x_1}{((x_1 - y_1)^2 + (x_2 - y_2)^{2})^2} \,dy_1\rt) dy_2 \\
		&\quad\leq Cx_1\| \theta \|_{\calC^{0,\alpha}} \int_{0}^{x_2} \frac {x_2 -y_2}{x_1^2 + (x_2 - y_2)^{2}} \,dy_2 \\
		&\quad\leq Cx_1 \| \theta \|_{\calC^{0,\alpha}} \log \frac {x_1^2+x_2^2}{x_1^2} .
	\end{aligned}
\end{equation*} 
Hence, we collect all the estimates up to now to attain \eqref{I1_est}.

Next, we estimate $II_1$. Due to the asymmetry of integrands around $y_2 = x_2$ on the interval $[0, 2x_2]$, we have 
\[
II_1 = - \theta(x) \int_{[0,\infty)\times [2x_2,\infty)} \frac {x_2-y_2}{|x - y|^{2+\alpha}} - \frac {x_2-y_2}{|x - \tilde{y}|^{2+\alpha}} \,dy.\]
Thanks to \eqref{II_1}, it suffices to show that 
	\begin{equation*}
	\begin{aligned}
		&\left| - \theta(x) \int_{[0,2x_1]\times [2x_2,\infty)} \frac {x_2-y_2}{|x - y|^{2+\alpha}} - \frac {x_2-y_2}{|x - \tilde{y}|^{2+\alpha}} \,dy \right| \\
		&\quad= \left| -\theta(x) \int_{[0,2x_1] \times [2x_2,\infty)}  \frac {(x_2 -y_2)(g(1)-g(0))}{|x - y|^{2+\alpha}|x - \tilde{y}|^{2+\alpha}} \,dy \right| \\
		&\quad\leq Cx_1\| \theta \|_{\calC^{0,\alpha}}.
	\end{aligned}
\end{equation*} 
On $[0, 2x_1] \times [2x_2 , \infty)$, we have $|x-y|^{\alpha} \geq (y_2-x_2)^{\alpha} \geq x_2^{\alpha}$. Then  we estimate this term as we did for $I_1$ on the region $[0,2x_1]\times [0,\infty)$:
\begin{equation*}
	\begin{aligned}
		&\left| -\theta(x) \int_{[0,2x_1] \times [2x_2,\infty)}  \frac {(x_2 -y_2)(g(1)-g(0))}{|x - y|^{2+\alpha}|x - \tilde{y}|^{2+\alpha}} \,dy \right| \\
		&\quad\leq Cx_1\| \theta \|_{\calC^{0,\alpha}} \int_{2x_2}^{\infty} \frac {x_1}{x_1^2 + (x_2-y_2)^2} \lt(\int_{0}^{x_1} \frac {x_2 -y_2}{(x_1 - y_1)^2 + (x_2 - y_2)^{2}} \,dy_1 \rt) dy_2 \\
		&\quad\leq Cx_1\| \theta \|_{\calC^{0,\alpha}} \int_{2x_2}^{\infty} \frac {x_1}{x_1^2 + (x_2-y_2)^2} \,dy_2 \\
		&\quad\leq Cx_1\| \theta \|_{\calC^{0,\alpha}},
	\end{aligned}
\end{equation*} which implies $$\left| II_1 \right| \leq Cx_1\| \theta \|_{\calC^{0,\alpha}}.$$
	
	For  $I_2$, we redefine the function $g$ by $g(t) = |y-(\bar{x} - (1-t)(\bar{x}+x))|^{2+\alpha}$. Then one has 
	\begin{align*}
		|\bar x - y|^{2+\alpha} - |x + y|^{2+\alpha} 
		&= g(1)-g(0) \\
		&= -2(2+\alpha)x_1(\tau(y_1+x_1) + (1-\tau)(y_1- x_1)) |\tau(y-\bar{x}) + (1-\tau)(y+x)|^{\alpha} 
	\end{align*} 
	for some $\tau \in (0,1)$ and with this we proceed similarly to the estimates for $I_1$ to get 
	\[
	\begin{aligned}
		\left| I_2 - 2(2+\alpha)x_1 \int_{Q(x)} \frac {y_1y_2}{|y|^{4+\alpha}} \theta(y) \,dy \right| \leq Cx_1\| \theta \|_{\calC^{0,\alpha}} \left( 1+ \log \frac {x_1^2 + x_2^2}{x_1^2} \right)
	\end{aligned}\]
	due to $|\bar x - y| = |x - \bar y|$. 
	
	For $II_2$,  we use $|g(1)-g(0)| \leq Cx_1(x_1+y_1)|x+y|^{\alpha}$, $|x+y| \geq |x|$, and $|x-\bar{y}| \geq x_2 + y_2$, on the region $[0,\infty) \times [0,2x_2]$ to have
\begin{equation*}
	\begin{aligned}
		&\left| -\theta(x) \int_{[0,\infty) \times [0,2x_2]}  \frac {(x_2 +y_2)(h(1)-h(0))}{|x + y|^{2+\alpha}|x - \bar{y}|^{2+\alpha}} \,dy \right| \\
		&\quad \leq C\frac {x_1}{|x|} \|\theta\|_{\calC^{0,\alpha}} \int_0^{2x_2} \lt( \int_0^{\infty}  \frac {(x_2 +y_2)}{(x_1-y_1)^2 + (x_2+y_2)^2} \,dy_1 \rt) dy_2 \\
		&\quad \leq C\frac {x_1}{|x|} \| \theta \|_{\calC^{0,\alpha}} \int_0^{2x_2} 1 \,dy_2 \\
		&\quad \leq Cx_1 \| \theta \|_{\calC^{0,\alpha}}.
	\end{aligned}
\end{equation*}
Since the other parts of the integral can be treated similarly to the estimates for $II_1$, we have
	\[
		\left| II_2 \right| \leq Cx_1\| \theta \|_{\calC^{0,\alpha}}.
	\]
Therefore, we combine all the above estimates to obtain \eqref{u1}.

	Let $u_2 = III_1 + III_2 + IV_1 + IV_2$, where	
	\begin{equation*}
		\begin{aligned}
			III_1 &:= \int_{\R^2} \left( \frac {x_1 -y_1}{|x - y|^{2+\alpha}} - \frac {x_1 -y_1}{|x - \bar{y}|^{2+\alpha}} \right) (\theta(y) -\theta(x)) \,dy, \\
			III_2 &:= \int_{\R^2} \left( \frac {x_1 + y_1}{|x + y|^{2+\alpha}} - \frac {x_1 + y_1}{|x - \tilde{y}|^{2+\alpha}} \right) (\theta(y) -\theta(x)) \,dy, \\
			IV_1 &:= \theta(x) \int_{\R^2} \left( \frac {x_1 -y_1}{|x - y|^{2+\alpha}} - \frac {x_1 -y_1}{|x - \bar{y}|^{2+\alpha}} \right) dy,  \quad \mbox{and}\\
			IV_2 &:= \theta(x) \int_{\R^2} \left( \frac {x_1 + y_1}{|x + y|^{2+\alpha}} - \frac {x_1 + y_1}{|x - \tilde{y}|^{2+\alpha}} \right) dy.
		\end{aligned}
	\end{equation*} 
	Then, we can show that 
	\[
	\begin{aligned}
	&\left| III_1 + 2(2+\alpha)x_2 \int_{Q(x)} \frac {y_1y_2}{|y|^{4+\alpha}} \theta(y) \,dy \right| \leq Cx_2\| \theta \|_{\calC^{0,\alpha}} \left( 1+ \log \frac {x_1^2 + x_2^2}{x_2^2} \right),\\
	&\left| III_2 + 2(2+\alpha)x_2 \int_{Q(x)} \frac {y_1y_2}{|y|^{4+\alpha}} \theta(y) \,dy \right| \leq Cx_2\| \theta \|_{\calC^{0,\alpha}} \left( 1+ \log \frac {x_1^2 + x_2^2}{x_2^2} \right),
	\end{aligned}
	\]
	 and
	 \[
	 \left| IV_1 \right| + \left| IV_2 \right| \leq Cx_2 \| \theta \|_{\calC^{0,\alpha}},\]
	 which yields \eqref{u2}. Since the proof is parallel to that of $u_1$ by changing the role of the first and the second components, we omit the details of that. This completes the proof.
\end{proof}

\subsection{Proof of Theorem~\ref{thm:inf}}\label{sec_pf}

Let $\varepsilon > 0$ and consider smooth initial data $\theta_0^{(N)} \in \calC^{\infty}_c(\mathbb{R}^2)$ given by \eqref{eq:nonexist-data} with $N \geq n_0$. Recalling Section~\ref{sec_ini}, we fix $n_0 > 0$ such that \eqref{small_ini} holds for all $N \geq n_0$. To derive a contradiction, suppose that there exists $M>0$ and $T>0$ such that each corresponding smooth solution $\theta^{(N)}(t)$ exists on the interval $[0,T]$ and
\begin{equation}\label{eq:sol-bound}
	\begin{split}
		\sup_{t \in [0,T]}\| \theta(t) \|_{ \calC^{0,\alpha}} \leq M.
	\end{split}
\end{equation}
Though it is well known that the smooth solution $\theta^{(N)}(t)$ actually exists at least on the interval $[0,T^{(N)}]$, it is unclear the lifespan of $\theta^{(N)}$ is enough to generate $\calC^{0,\alpha}$--norm inflation (note that $T^{(N)} \to 0$ as $N \to \infty$). This is why the lifespan of $\theta^{(N)}(t)$ is included in our assumptions. For simplicity, we assume that $M\ge 1$ and $MT \ll 1$. Since $\theta_0^{(N)}$ satisfies the odd-odd symmetry, the uniqueness result in Section~\ref{uniq_holder} guarantees that $\theta^{(N)}(t)$ should be odd-odd symmetric on $[0,T]$.

We recall from Lemma \ref{log_lip_gen} that 
\[
|u(t,x) - u(t,y)| \lesssim M |x-y| \left( 10 + \log \frac {1}{|x-y|} \right).
\]
 Then, the forward chateristic flow $\Phi(t;0,\cdot) : \R^2 \to \R^2$ defined as a solution to  
 \[
 \frac {d}{dt} \Phi(t;0,x) = u(t,\Phi(t;0,x)), \quad \Phi(0;0,x) = x
 \] 
 is well-defined, and there exists a constant $C>0$ such that for any $t \in [0,T]$,
 \[
 |x-y|^{\exp(CMt)} \le |\Phi(t;0,x)-\Phi(t;0,y)| \le |x-y|^{\exp(-CMt)}
 \] 
 holds for all $x$, $y \in \R^2$ with $|x-y|\le 1$ and 
 \[
|x-y|e^{-CMt}  \le |\Phi(t;0,x)-\Phi(t;0,y)| \le |x-y| e^{CMt}
 \] 
 holds for all $x$, $y \in \R^2$ with $|x-y|> 1$.
  This implies $\operatorname{supp} \theta(t,\cdot) \subset B(0,R)$ on $[0,T]$ for some $R>0$. Moreover, due to the odd-odd symmetry and the continuity of $\theta(t,\cdot)$, the solution vanishes at the axes and it is conserved along the characteristics:
\[
\theta(t,\Phi(t;0,x)) = \theta_0(x), \qquad x \in \R^2\,\, \mbox{and}\,\, t \in [0,T].
\]
We also observe that the solution $\theta$ satisfies the assumptions in Lemma~\ref{key_lem}. Thus we deduce from \eqref{u1} that
\begin{equation}\label{phi1}
	\begin{aligned}
		\left| \frac {d}{dt} \log \Phi_1(t;0,x) - 4(2+\alpha)\int_{Q(\Phi(t;0,x))} \frac { y_1 y_2 }{|y|^{4+\alpha}} \theta(y) \,dy\right| \le CM \log \left( e + \frac{\Phi_2(t;0,x)}{\Phi_1(t;0,x)} \right)
	\end{aligned}
\end{equation} and
\begin{equation}\label{phi2}
	\begin{aligned}
		\left| \frac {d}{dt} \log \Phi_2(t;0,x) + 4(2+\alpha)\int_{Q(\Phi(t;0,x))} \frac { y_1 y_2 }{|y|^{4+\alpha}} \theta(y) \,dy\right| \le CM \log \left( e + \frac{\Phi_1(t;0,x)}{\Phi_2(t;0,x)} \right)
	\end{aligned}
\end{equation}
for any $x \in \R^2$ with $x_1, x_2>0$. Here $C>0$ is independent of $M$.

\begin{lemma}\label{lem_back}
	For any given $\gamma>1$ and $K \geq 1$, there exists $T=T(M,\gamma,K) > 0$ such that
	\begin{equation}\label{back}
	\begin{aligned}
		\frac {\Phi_2(t;0,x)}{\Phi_1(t;0,x)} \leq \gamma \frac {x_2}{x_1}, \qquad t \in [0,T]
	\end{aligned}
\end{equation} for all $x \in \R^2$ with $x_1>0$, $x_2>0$, and $\frac {1}{K} \leq \frac {x_2}{x_1} \leq K$.
\end{lemma}
\begin{proof}
From \eqref{phi1} and \eqref{phi2}, we have $$\frac {d}{dt} \log \frac {\Phi_2(t;0,x)}{\Phi_1(t;0,x)} \leq CM \left( \log \left( e + \frac{\Phi_2(t;0,x)}{\Phi_1(t;0,x)} \right) + \log \left( e + \frac{\Phi_1(t;0,x)}{\Phi_2(t;0,x)} \right) \right).$$ 
Note that \eqref{back} holds for $t=0$. Assume that there exists $[t_1,t_2] \subset [0,T]$ such that $\frac {\Phi_2(t;0,x)}{\Phi_1(t;0,x)} \geq \frac {\gamma}{2} \frac {x_2}{x_1}$ on $[t_1,t_2]$. Then, it follows that 
\[
\frac {d}{dt} \log \frac {\Phi_2(t;0,x)}{\Phi_1(t;0,x)} \leq CM + CM \log \frac {\Phi_2(t;0,x)}{\Phi_1(t;0,x)}.
\]
Thus, by Gr\"onwall's lemma, we have 
	\[
	\log \frac {\Phi_2(t;0,x)}{\Phi_1(t;0,x)} \leq e^{CMT} \lt(\log \lt(\frac {\gamma}{2} \frac {x_2}{x_1}\rt) + CMT\rt), \qquad t \in [t_1,t_2].
	\] 
	By taking $T>0$ sufficiently small and using the classical continuation argument, we obtain the desired estimate.
\end{proof}

Now, let $\Omega_n := \operatorname{supp} \theta_{0,loc}^{(n)}$ for $n \geq n_0$, and take $\gamma>1$. Since $ x=(x_1, x_2) \in \Omega_n$ satisfies
\[
\frac13 = \frac{4^{-n-2} - 4^{-n}\cdot \frac{1}{32}}{4^{-n-2} + 4^{-n}\cdot \frac{1}{32}} \le \frac{x_2}{x_1} \le \frac{4^{-n-2} + 4^{-n}\cdot \frac{1}{32}}{4^{-n-2} - 4^{-n}\cdot \frac{1}{32}} = 3,
\]
we get \eqref{back} for any  $x \in \Omega_n$ with $n \geq n_0$. By combining this with Remark~\ref{rmk_R}, we have
\begin{equation}\label{phi1'}
	\begin{aligned}
		\left| \frac {d}{dt} \log \Phi_1(t;0,x) - 4(2+\alpha)\int_{R(\Phi(t;0,x))} \frac { y_1 y_2 }{|y|^{4+\alpha}} \theta(y) \,dy\right| \le CM, \qquad x \in \Omega_n.
	\end{aligned}
\end{equation} Similarly, we can obtain
\begin{equation}\label{phi2'}
	\begin{aligned}
		\left| \frac {d}{dt} \log \Phi_2(t;0,x) + 4(2+\alpha)\int_{R(\Phi(t;0,x))} \frac { y_1 y_2 }{|y|^{4+\alpha}} \theta(y) \,dy\right| \le CM \log \left( e + \frac{\Phi_1(t;0,x)}{\Phi_2(t;0,x)} \right), \qquad x \in \Omega_n.
	\end{aligned}
\end{equation}
From \eqref{phi1'}, by employing the analogous argument to \cite[Claim I]{JK24} or \cite[Lemma~4.1]{JK22}, we deduce
\begin{equation}\label{order}
	\begin{aligned}
		2\sup_{x \in \cup_{n > N} \Omega_n} \Phi_1(t;0,x) \leq \inf_{x \in \Omega_N} \Phi_1(t;0,x), \qquad t \in [0,T]
	\end{aligned}
\end{equation} for any $N \geq n_0$.

%For simplicity, let 
%\[
%K_n(t) := \int_{\Phi(t;0,\Omega_n)} \frac {y_1y_2}{|y|^{4+\alpha}} \theta(t,y) \,dy = \int_{\Omega_n} \frac {\Phi_1(t;0,y) \Phi_2(t;0,y)}{|\Phi(t;0,y)|^{4+\alpha}} \theta_0(y) \,dy, \qquad n \geq n_0.
%\]

In the lemma below, we present the comparability between Eulerian and Lagrangian variables. 

\begin{lemma}
	Let $\gamma>1$. Then for each $n_0 \leq n \leq N$, there exists $T_{n} \in (0,T]$ with $T_{n} \gtrsim \frac{1-\beta}{n^{1-\beta}}$ for sufficiently large $n$ such that 
	\begin{equation}\label{eqv_est}
	\begin{aligned}
		\frac {1}{\gamma} \leq \frac {\Phi_1(t;0,x)}{x_1} \leq \gamma, \qquad \frac {1}{\gamma} \leq \frac {\Phi_2(t;0,x)}{x_2} \leq \gamma, \qquad t \in [0,T_n]
	\end{aligned}
\end{equation} for any $x \in \Omega_n$.
\begin{proof}
	We proceed by induction argument and start with $n = n_0$. For every $x \in \Omega_n$, we first deduce from \eqref{order} that $\theta(t,y) = 0$ for any $t \in [0,T]$ and $y \in R(\Phi(t;0,x))$. We combine this with \eqref{phi1'} and \eqref{phi2'} to get 
	\[
	\frac {d}{dt} \log \frac {\Phi_1(t;0,x)}{\Phi_2(t;0,x)} \leq CM \log \left( e + \frac{\Phi_1(t;0,x)}{\Phi_2(t;0,x)} \right).
	\]
	 Then applying Gr\"onwall's lemma yields
	 \[
	 \frac {\Phi_1(t;0,x)}{\Phi_2(t;0,x)} \leq C, \qquad t \in [0,T].
	 \]
	  Thus, we apply this to \eqref{phi1'} and \eqref{phi2'}  and obtain
	  \[
	  \left| \frac {d}{dt} \log \Phi_1(t;0,x) \right| + \left| \frac {d}{dt} \log \Phi_2(t;0,x) \right| \leq CM.
	  \]
	  By taking $T_{n_0} = T$ and using $T M \ll 1$, we obtain \eqref{eqv_est} with $n=n_0$.
	
	Suppose that \eqref{eqv_est} holds for $n=n_0, \dots, N-1$ with the corresponding time $T_{n}$. Let $n=N$. From \eqref{phi1'} and \eqref{phi2'}, we have \begin{equation*}
	\begin{aligned}
		\frac {d}{dt} \log \frac {\Phi_1(t;0,x)}{\Phi_2(t;0,x)} &\leq C\int_{R(\Phi(t;0,x))} \frac { y_1 y_2 }{|y|^{4+\alpha}} \theta(y) \,dy + CM \log \left( e + \frac{\Phi_1(t;0,x)}{\Phi_2(t;0,x)} \right)\\
%		&= C\sum_{m < N} K_m + CM \log \left( e + \frac{\Phi_1(t;0,x)}{\Phi_2(t;0,x)} \right) \\
		&= C \sum_{m<N}\int_{\Omega_m} \frac {\Phi_1(t;0,y) \Phi_2(t;0,y)}{|\Phi(t;0,y)|^{4+\alpha}} \theta_0(y) \,dy + CM \log \left( e + \frac{\Phi_1(t;0,x)}{\Phi_2(t;0,x)} \right) \\
		&\leq C \sum_{m<N} \int_{\Omega_m} \frac {y_1y_2}{|y|^{4+\alpha}} \theta_0(y) \,dy + CM \log \left( e + \frac{\Phi_1(t;0,x)}{\Phi_2(t;0,x)} \right) \\
		&= C \sum_{m<N} m^{-\beta} + CM \log \left( e + \frac{\Phi_1(t;0,x)}{\Phi_2(t;0,x)} \right),
	\end{aligned}
\end{equation*} where we used \eqref{order} and the induction hypothesis at the last inequality. We now take 
\[
T_N = \min \lt\{ T, \ \frac{C}{\sum_{m < N} m^{-\beta}} \rt\}. 
\]
By a similar argument, we also obtain
\[
\log \frac {\Phi_1(t;0,x)}{\Phi_2(t;0,x)} \leq C, \qquad t \in [0,T_N],
\]
and this concludes our desired estimate.
\end{proof}
\end{lemma}

Now, we are ready to prove Theorem~\ref{thm:inf}. We take any $x \in \Omega_n$ with $n \gg n_0$ and combination of \eqref{phi1'} and \eqref{phi2'} gives
\begin{equation*}
	\begin{aligned}
		\frac {d}{dt} \log \frac{\Phi_1(t;0,x)}{\Phi_2(t;0,x)} &\leq 8(2+\alpha)\int_{R(\Phi(t;0,x))} \frac {y_1y_2}{|y|^{4+\alpha}} \theta(t,y) \,d y + CM \log \left( e+ \frac {\Phi_1(t;0,x)}{\Phi_2(t;0,x)} \right) \\
		&\leq C\int_{R(\Phi(t;0,x))} \frac {y_1y_2}{|y|^{4+\alpha}} \theta(t,y) \,d y + CM \log \left( e+ \frac {\Phi_1(t;0,x)}{\Phi_2(t;0,x)} \right).
	\end{aligned}
\end{equation*} If $\frac {\Phi_1(t;0,x)}{\Phi_2(t;0,x)} \leq e$, we find
\[
\frac {d}{dt} \log \frac{\Phi_1(t;0,x)}{\Phi_2(t;0,x)} \leq C\int_{R(\Phi(t;0,x))} \frac {y_1y_2}{|y|^{4+\alpha}} \theta(t,y) \,d y + CM,
\] 
and if $\frac {\Phi_1(t;0,x)}{\Phi_2(t;0,x)} \geq e$, we get
\[
\frac {d}{dt} \log \frac{\Phi_1(t;0,x)}{\Phi_2(t;0,x)} \leq C\int_{R(\Phi(t;0,x))} \frac {y_1y_2}{|y|^{4+\alpha}} \theta(t,y) \,d y + CM \log \frac {\Phi_1(t;0,x)}{\Phi_2(t;0,x)}.
\]
Applying Gr\"onwall's lemma to the above implies
\[
\log \frac{\Phi_1(t;0,x)}{\Phi_2(t;0,x)} \leq \left( \log \frac {x_1}{x_2} + C\int_0^t \int_{R(\Phi(t;0,x))} \frac {y_1y_2}{|y|^{4+\alpha}} \theta(t,y) \,d y dt + CMt \right) e^{CMt}.
\] 
Hence, we deduce from \eqref{phi2'} that 
\begin{equation*}
	\begin{aligned}
		\frac {d}{dt} \log \Phi_2(t;0,x) &\leq -4(2+\alpha) \int_{R(\Phi(t;0,x))} \frac {y_1y_2}{|y|^{4+\alpha}} \theta(t,y) \,d y + CM \log \left( e + \frac {\Phi_1(t;0,x)}{\Phi_2(t;0,x)} \right) \\
		&\leq -4(2+\alpha) \int_{R(\Phi(t;0,x))} \frac {y_1y_2}{|y|^{4+\alpha}} \theta(t,y) \,d y + C\int_0^t \int_{R(\Phi(t;0,x))} \frac {y_1y_2}{|y|^{4+\alpha}} \theta(t,y) \,d y dt + CM.
	\end{aligned}
\end{equation*}
This yields
 \begin{equation*}
	\begin{aligned}
		\log \frac{\Phi_2(T_{\sqrt{n}},x)}{x_2} &\leq -4(2+\alpha) \int_0^{T_{\sqrt{n}}} \int_{R(\Phi(t;0,x))} \frac {y_1y_2}{|y|^{4+\alpha}} \theta(t,y) \,d y dt \cr
		&\quad + CT\int_0^{T_{\sqrt{n}}}\int_{R(\Phi(t;0,x))} \frac {y_1y_2}{|y|^{4+\alpha}} \theta(t,y) \,d y dt + CMT \\
		&\leq -2 \int_0^{T_{\sqrt{n}}} \int_{R(\Phi(t;0,x))} \frac {y_1y_2}{|y|^{4+\alpha}} \theta(t,y) \,d y dt + CMT.
	\end{aligned}
\end{equation*} 
Due to the nonnegativity of $\theta$ on $R(\Phi(t;0,x))$, we see that 
\begin{equation*}
	\begin{aligned}
		-2 \int_0^{T_{\sqrt{n}}} \int_{R(\Phi(t;0,x))} \frac {y_1y_2}{|y|^{4+\alpha}} \theta(t,y) \,d y dt &\leq -2 \sum_{\sqrt{n}<k<n} \int_0^{T_k} \int_{\Phi(t;0,\Omega_k)} \frac {y_1y_2}{|y|^{4+\alpha}} \theta(t,y) \,d y dt \\
		&\leq - \sum_{\sqrt{n}<k<n} ck^{-\beta} T_k \\
		&\leq - c(1-\beta) \log n,
	\end{aligned}
\end{equation*}
where we used the lower bound of $T_k$ at the last inequality. Thus, we obtain
\[
\frac {\Phi_2(T_{\sqrt{n}},x)}{x_2} \leq C n^{-c(1-\beta)}.
\]
Here, we observe that 
 \[
 \| \theta(T_{\sqrt{n}},\cdot) \|_{\calC^{0,\alpha}} \geq \frac {\theta(T_{\sqrt{n}},\Phi(T_{\sqrt{n}},x)) - \theta(T_{\sqrt{n}},\Phi_1(T_{\sqrt{n}},x),0)}{\Phi_2(T_{\sqrt{n}},x)^{\alpha}} = \frac {\theta_0(x)-\theta_0(x_1,0)}{x_2^{\alpha}} \left( \frac {x_2}{\Phi_2(T_{\sqrt{n}},x)} \right)^{\alpha}.
 \]
  We fix $x = (4^{-n-2},4^{-n-2})$. Since 
  \[
  \frac {\theta_0(x)-\theta_0(x_1,0)}{x_2^{\alpha}} \simeq n^{-\beta},
  \]
  we can get 
  \[
  \frac {\theta_0(x)-\theta_0(x_1,0)}{x_2^{\alpha}} \left( \frac {x_2}{\Phi_2(T,x)} \right)^{\alpha} \gtrsim n^{-\beta} n^{c\alpha(1-\beta)} \gtrsim n^{c\alpha - (1+c\alpha)\beta}.
  \] 
  Finally, we take sufficiently small $\beta$ such that $c\alpha-(1+c\alpha)\beta =: \eta > 0$ and consider $N$ with $N \geq n_0^2$, $N \geq M^{2/\eta}$ and $T_{\sqrt{N}} \simeq \frac{1-\beta}{N^{1-\beta}} \leq T$. This yields $\| \theta(T_{\sqrt{N}},\cdot) \|_{\calC^{0,\alpha}} \geq M^2$ that contradicts \eqref{eq:sol-bound}. This completes the proof for Theorem~\ref{thm:inf}.

%%%%%%%%%%%%%%%%%%%%%%%%%%%%%%%%%%%%%%%%%%%%%%%
%
%
%
%
%
%
%%%%%%%%%%%%%%%%%%%%%%%%%%%%%%%%%%%%%%%%%%%%%%%

\section*{Acknowledgments}
The work of Y.-P. Choi is supported by NRF grant no. 2022R1A2C1002820. The work of J. Jung is supported by the research fund of Hanyang University (HY-202400000001275). %The work of J. Kim was supported by 

\section*{Conflict of interest}
The authors declare that they have no conflict of interest.

\end{document}